\documentclass[preprint]{elsarticle}
\usepackage{amsfonts,amssymb,mathrsfs,amsmath,amscd,epsfig,setspace,amstext,verbatim}
\usepackage{paralist,hyperref}

\newenvironment{proof}{\medskip                    
\noindent{\scshape Proof:}}{\quad $\square$
\medskip}  

\usepackage{hyperref}
\newtheorem{theorem}{Theorem}[section]

\newtheorem{proposition}{Proposition}[section]
\newtheorem{corollary}{Corollary}[section]

\newtheorem{algorithm}{Algorithm}[section]
\newtheorem{definition}{Definition}[section]

\newcommand{\arxiv}[1]{\href{http://www.arXiv.org/abs/#1}{arXiv:#1}}
\usepackage{graphicx,color,graphics}
\usepackage{latexsym}

\newcommand{\myinf}{\mathop{\text{\Large$\wedge$}}}

\def\ltr{\text{``}}
\def\rtr{\text{''}}

\def\digr{{\mathcal D}}

\def\Rmax{\mathbb{R}_{\max}}

\def\R{\mathbb{R}}

\def\Rnn{\R^{n\times n}}

\def\bunity{\mathbf{1}}

\def\rad{\rho}

\def\tplus{\oplus}
\def\lplus{\tplus}
\def\myplus{\tplus}
\def\bigtplus{\bigoplus}
\def\tmult{\otimes}
\def\lmult{\tmult_L}
\def\bzero{\mathbf{0}}
\def\crit{{\mathcal C}}
\def\unint{[0,1]}
\def\rad{\rho}
\def\amine{A^{(1)}}
\def\aminle{A^{(\lambda)}}
\def\aminlambda{\aminle}
\def\aminre{A^{(\rad)}}
\def\tconv{\operatorname{tconv}}
\def\Schur{\operatorname{Schur}}
\def\otimesluk{\otimes_L}
 
\begin{document}

\title{Tropical linear algebra with the {\L}ukasiewicz T-norm}

\author[rvt]{Martin Gavalec\fnref{fn1}}
\ead{martin.gavalec@uhk.cz}

\author[rvt]{Zuzana N\v{e}mcov\'{a}\fnref{fn1}}
\ead{zuzana.nemcova@uhk.cz}

\author[rvt2]{Serge{\u\i} Sergeev\corref{cor}\fnref{fn2}}
\ead{sergiej@gmail.com}

\address[rvt]{University of Hradec Kr\'{a}lov\'{e}/Faculty of Informatics and Management, Rokitansk\'{e}ho 62,
Hradec Kr\'{a}lov\'{e} 3, 500 03, Czech Republic}
\address[rvt2]{University of Birmingham, School of Mathematics, Edgbaston B15 2TT}

\cortext[cor]{Corresponding author. Email: sergiej@gmail.com}
\fntext[fn1]{Supported by the Czech Science Foundation project \# 14-02424S and Grant Agency of
Excellence UHK FIM \# 2204.}

\fntext[fn2]{Supported by EPSRC grant EP/J00829X/1, RFBR grant
12-01-00886. This research was
initiated when the third author was with INRIA and CMAP \'Ecole
Polytechnique, 91128 Palaiseau Cedex, France.}

\begin{abstract}
The max-{\L}ukasiewicz semiring is defined as the unit interval
$[0,1]$ equipped with the arithmetics $\ltr a+b\rtr=\max(a,b)$ and
$\ltr ab\rtr=\max(0,a+b-1)$. Linear algebra over this semiring can
be developed in the usual way. We observe that any 
problem of the max-{\L}ukasiewicz linear algebra can be equivalently formulated
as a problem of the tropical (max-plus) linear algebra.
Based on this equivalence, we
develop a theory of the matrix powers and the eigenproblem over the
max-{\L}ukasiewicz semiring.
\end{abstract}

\begin{keyword}
tropical, max-plus, {\L}ukasiewicz, eigenvector, matrix power \vskip0.1cm {\it{AMS
Classification:}} 15A80, 15A06, 15A18
\end{keyword}

\maketitle

\section{Introduction}\label{s:introduction}
The max-{\L}ukasiewicz semiring is defined over the interval $[0,1]$
equipped with the operations of addition $a\oplus b=\max(a,b)$ and
multiplication $a\otimesluk b=\max(0,a+b-1)$.  These operations are
extended to matrices and vectors in the usual way: $(A\tplus
B)_{ij}=a_{ij}\tplus b_{ij}$ and $(A\otimesluk B)_{ik}=\bigtplus_j
a_{ij}\otimesluk b_{jk}$. We consider the
max-{\L}ukasiewicz powers of matrices $A^{\otimesluk k}
:=\overbrace{A\otimesluk\cdots\otimesluk A}^{k},$ and the spectral
problem over max-{\L}ukasiewicz semiring: given $\lambda\in[0,1]$,
find $x\in[0,1]^n$ with not all components $0$ such that
$A\otimesluk x =\lambda\otimesluk x$.

Our study of max-{\L}ukasiewicz semiring is motivated by the recent
success of tropical linear algebra, developed over the max-plus
semiring $\Rmax=\R\cup\{-\infty\}$ equipped with operations of
``addition'' $a\oplus b=\max(a,b)$ and ``multiplication'' $a\otimes
b=a+b$. We convert the problems of max-{\L}ukasiewicz linear
algebra, i.e., the linear algebra over max-{\L}ukasiewicz semiring,
to the problems of tropical (max-plus) linear algebra. Then we can take
advantage of the well-developed theory and algorithms of the latter,
see~\cite{ABG-06,BCOQ, But:10,  BSS-12, CG:79, GK-07, GKS-12}, to
mention only a few possible sources. Our main ingredients are the
theory of spectral problem $A\otimes x=\lambda\otimes x$, Bellman
equations or Z-equations $x=A\otimes x\oplus b$, one-sided systems
$A\otimes x=b$ and, occasionally, the tropical linear programming.
Recall that the basic theory of one-sided systems $A\otimes x=b$ involves
residuation theory~\cite{CC-95} as well as set coverings~\cite{But:10}.
Also see Rashid et al.~\cite{RGS-12} for a preliminary work on the
max-{\L}ukasiewicz linear algebra exploring details of the
three-dimensional case.

A basic idea behind this paper is that each fuzzy triangular norm
(t-norm) as described, for instance, in Klement, Mesiar and
Pap~\cite{KMP:00}, leads to an idempotent semiring, which we call a
max-t semiring. There are directions of abstract fuzzy sets theory
which are related to the present work. Some connections of
multivalued logic and fuzzy algebra with idempotent mathematics have
been developed by Di Nola et al.~\cite{DNG-05,DN+07,DNR-07,DNR-13}. These works develop certain
aspects of algebra over semirings arising from fuzzy logic
(MV-algebras, {\L}ukasiewicz transform), which currently seem most
interesting and useful for the fuzzy sets theory.

However, neither general MV-algebras~\cite{DN+07} nor even the special case of 
{\L}ukasiewicz MV-algebra are considered in this paper.  We are rather
motivated by the basic problems of the tropical linear algebra,
which we are going to consider here in the context of
max-{\L}ukasiewicz semiring. We also remark that max-{\L}ukasiewicz
semiring can be seen as a special case of incline algebras of Cao,
Kim and Roush~\cite{CKR:84}, see also, e.g., Han-Li~\cite{HL-04} and
Tan~\cite{Tan-11}. One of the main problems considered in that
algebra is to study the periodicity of matrix powers over a larger
class of semirings, using lattice theory, lattice ordered groups and
residuations.  Let us also recall the distributive lattices as another
special case of incline algebras, although this special case does not include the
max-{\L}ukasiewicz semiring. Powers of matrices over distributive lattices
are studied, e.g., by Cechl\'arov\'a~\cite{Cech}.

The approach which we develop here, does not apply to
incline algebras in general (or to distributive lattices in particular), but it allows to study the
linear-algebraic problems over max-{\L}ukasiewicz algebra in much
more detail.

{\bf Aggressive network.} Let us also recall the basic network
motivation to study tropical (max-plus) linear algebra and, more generally,
linear algebra over semirings (Gondran-Minoux~\cite{GM:08}, see also
Litvinov-Maslov~\cite{LM-98}). This motivation suggests a directed
graph $\digr$ with nodes $N=\{1,\ldots,n\}$ and edges
$E=\{(i,j)\colon i,j\in N\}$, where each edge $(i,j)$ is weighted by
$a_{ij}$. Sarah is an agent travelling in the network. She is given
$1$ unit of money before entering it at node $i$ (say, $1$ thousand
of GBP), and $a_{ij}$ expresses the amount of money left on her bank
account after she moves from $i$ to $j$. The quantity
$c_{ij}:=1-a_{ij}$ expresses the cost of moving from $i$ to $j$.
More generally, if Sarah is given $x_i$ units of money, with $0\leq
x_i\leq 1$, then there are two cases: when $x_i-c_{ij}\geq 0$ and
when $x_i-c_{ij}<0$. In the first case, $x_i-c_{ij}=x_i+a_{ij}-1$
will be the money left on her account after she moves from $i$ to
$j$. In the second case, Sarah's account will be frozen, in other
words, her balance will be set to $0$ forever and with no excuse. In
any case, $x_i\otimesluk a_{ij}$ expresses the amount of money left
on Sarah's account if she goes from $i$ to $j$. More generally, if
Sarah is given $1$ unit of money and follows a walk
$P=(i_1,\ldots,i_k)$ on $\digr$, then the {\L}ukasiewicz weight of
$P$ computed as $w_L(P)=a_{i_1i_2}\otimesluk\cdots\otimesluk
a_{i_{k-1i_k}}$ will show how much money will remain on her bank
account. Computing matrix powers over {\L}ukasiewicz semiring, it
can be seen that the entry $(A^{\otimes_L t})_{ij}$ shows Sarah's
funds at $j$ if she chooses an optimal walk from $i$ to $j$.
Computing the left orbit of a vector, $(x\otimesluk A^{\otimes_L
t})_i$ also shows Sarah's funds if 1) for each $\ell$, $x_{\ell}$ is
the amount of money given to her if she enters the network in state
$\ell$, 2) she chooses an optimal starting node and an optimal walk
from that node to $i$, where the {\L}ukasiewicz weight of a walk
$P=(i_1,\ldots,i_k)$ is now computed as $x_{i_1}\otimesluk
a_{i_1i_2}\otimesluk\cdots\otimesluk a_{i_{k-1}i_k}$.

In the context of aggressive network, we can pose the {\L}ukasiewicz spectral problem
$x\otimesluk A=\lambda\otimesluk x$ if we want to control the dynamics of Sarah's funds, for instance, to
know precisely when the game will be over. As $(\lambda\otimesluk x)_i=\max(\lambda-1+x_i,0)$, it is also natural to
take a partition $(K,L)$ of $\{1,\ldots,n\}$, that is, the subsets $K,L\subseteq\{1,\ldots,n\}$ such that
$K\cup L=\{1,\ldots,n\}$ and $K\cap L=\emptyset$, and impose that $x_i\leq 1-\lambda$ for $i\in L$ and
$x_i\geq 1-\lambda$ for $i\in K$. The {\L}ukasiewicz eigenvectors satisfying these conditions are called
$(K,L)$-{\L}ukasiewicz eigenvectors. As we shall see, when $K$ and $L$ are proper subsets of $\{1,\ldots,n\}$,
the existence of $(K,L)$ eigenvectors is equivalent to $(K,L)$ being a ``secure partition'' of the network
where we subtract $\lambda$ from each edge, see Definition~\ref{def:secure}. This establishes a connection
between the {\L}ukasiewicz spectral problem and the combinatorics of weighted digraphs. The network sense of
$(K,L)$-eigenvectors is also clear if we require the strict inequalities $x_i< 1-\lambda$ for $i\in L$ and
$x_i> 1-\lambda$ for $i\in K$: in this case we know how much Sarah should be given in each state, in order that that the game will be over after one step if Sarah starts in any node of $L$, and in order that she has a chance to live longer if she starts in a node of $K$ and chooses an optimal trajectory.

For convenience, in the paper we will consider {\bf right} {\L}ukasiewicz eigenvectors and orbits.

\begin{figure}
 \begin{center}
  \includegraphics[width=320pt]{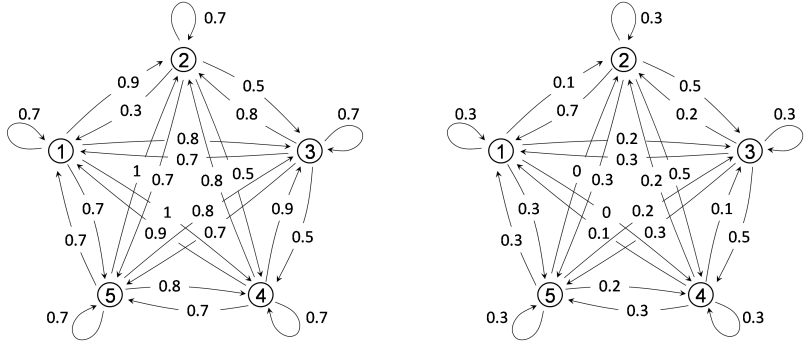}
  \caption{An example of aggressive network: values of travel costs (left) and account
  balance after one step (right)}\label{f:network}
  \end{center}
\end{figure}
Figure~\ref{f:network} represents an example of aggressive network.
The weights of edges on the left denote Sarah's debts, or payments
if she is able to make them, and the weights on the right stand for
her balance after moving along the corresponding edge.

The rest of the paper is organized as follows. Section~2 is occupied
with necessary preliminaries on tropical convexity and tropical
linear algebra. Section~3 develops the theory of
$(K,L)$-{\L}ukasiewicz eigenvectors and secure partitions. Here we
give an explicit description of generating sets of
$(K,L)$-{\L}ukasiewicz eigenspaces, establish the connection with
the problem of finding secure partitions of weighted digraphs, and
provide an algorithm for enumeration of all possible secure
partitions. In Section 4 we examine the powers and orbits of
matrices in {\L}ukasiewicz semiring, whose theory is closely related
to its well-known tropical (max-plus) counterpart.

\section{Basics of the tropical linear algebra}
\label{s:tropical}

The main idea of this paper is that the max-{\L}ukasiewicz linear
algebra is closely related to the tropical (max-plus) linear
algebra. A key observation relating the {\L}ukasiewicz linear
algebra with the tropical linear algebra is that, given a matrix
$A\in\unint^{m\times n}$ and a vector $x\in\unint^n$, we have
\begin{equation}
\label{convert}
A\lmult x=\amine\otimes x \myplus\bzero,
\end{equation}
where $\amine$ denotes the matrix with entries $a_{ij}-1$, and $\bzero$ is the vector with $m$ entries all equal
to $0$. Note that the entries of $\amine$ belong to $[-1,0]$, while the entries of $x$ are required to be in $[0,1]^n$.
More generally, we denote by $A^{(\alpha)}$
the matrix with entries $a_{ij}-\alpha$.

Sometimes we will also use the min-plus version of the tropical linear algebra, in particular, the min-plus addition
$a\wedge b:=\min(a,b)$ and the min-plus matrix product $(A\otimes' B)_{ik}=\myinf_j a_{ij}+b_{jk}$.
Denoting by $A^{\sharp}$ the matrix with entries $(-a_{ji})$, known as {\em Cuninghame-Green inverse},
one obtains the following duality law:
\begin{equation}
\label{e:resid}
A\otimes x\leq b \Leftrightarrow x\leq A^{\sharp}\otimes' b
\end{equation}

In the sequel we usually omit the $\otimes$ sign for tropical matrix
multiplication, unlike the $\otimesluk$ sign for the
max-{\L}ukasiewicz multiplication. We are only interested in
matrices and vectors with real entries, that is, with no $-\infty$
entries.

We note here that we only need the special case of tropical(max-plus)
semiring here. See, e.g., Gondran and Minoux~\cite{GM:08,GM-07} for more
general idempotent algebras (dio\"{\i}ds).

\subsection{Elements of tropical convexity}

A set $C\in(\R\cup\{-\infty\})^n$ is called tropically convex if together with any two points $x,y\in C$ it contains the whole
tropical segment
\begin{equation}
\label{tsegment}
[x,y]_{\oplus}:=\{\lambda x\oplus\mu y\colon \lambda\tplus\mu=0\}.
\end{equation}
Note that $x\oplus y\in[x,y]_{\oplus}$.

The {\bf tropical convex hull} of a set $X\subseteq (\R\cup\{-\infty\})^n$ is defined as
\begin{equation}
\label{tconv-def}
\tconv(X):=\{\bigoplus_{\mu} \lambda_{\mu} x_{\mu}\colon \bigoplus_{\mu}\lambda_{\mu}=0, x_{\mu}\in X\},
\end{equation}
where only a finite number of $\lambda_{\mu}$ are not $-\infty$. In this case, the {\em tropical
Carath\'eodory theorem} states that in~\eqref{tconv-def}, we can restrict without loss of generality 
to tropical convex combinations of no more than
$n+1$ points $x_{\mu}$.

As in the case of the usual convexity, there exists an {\em internal
description} of tropical convex sets in terms of extremal points and
recessive rays, and {\em external description} as intersection of
(tropical) halfspaces. See \cite{All:09,CGQS-05,DS-04,GK-07} for
some of the recent references. Here we will be mostly interested in
the internal description of a compact tropical convex set $C$ based
on {\em extreme points}: if represented as a point in a tropical
segment of $C$, such a point should coincide with one of its ends.

It follows that the linear equations over max-{\L}ukasiewicz
semiring are affine equations over max-plus semiring, with the
solutions confined in the hypercube $[0,1]^n$. The solution sets to
systems of such equations are compact and tropically convex.
Moreover, they are {\em tropical polyhedra}, i.e. tropical convex
hulls of a finite number of points.

A tropical analogue of a theorem of Minkowski was proved in full
generality by Gaubert and Katz~\cite{GK-07}, see also Butkovi\v{c}
et al.~\cite{BSS} for a part of this result. Here we are interested
in the particular case of compact tropically convex sets in $\R^n$.

\begin{theorem}
\label{t:mink}
Let $C\subseteq\R^n$ be a compact tropical convex set.
Further let $u^{(\mu)}$ be a (possibly infinite) set of its extreme points. Then,
\begin{equation}
\label{e:mink}
C=\left\{\bigoplus_{\mu} \lambda_{\mu}u^{(\mu)}, \lambda_{\mu}\in\R\cup\{-\infty\},\ \bigoplus_{\mu} \lambda_{\mu}=0\right\}.
\end{equation}
where in~\eqref{e:mink}, only a finite number of $\lambda_{\mu}$ are not equal to $-\infty$.
In words, any compact tropically convex subset of $\R^n$ is generated by its extreme points.
\end{theorem}

\subsection{Cyclicity theorem}
\label{ss:cyclicity} Starting from this subsection, all matrices
have only finite entries, no $-\infty$. We now consider the sequence
of max-plus matrix powers $A^k=\overbrace{A\otimes\cdots\otimes
A}^{k}$, for $A\in\Rnn$.

The max-algebraic {\em cyclicity theorem} states that if the maximum
cycle mean of $A\in\Rnn$  (to be defined later) equals $0$, then the
sequence of max-plus matrix powers $A^k$ becomes periodic after some
finite transient time $T(A)$, and that the ultimate period of $A^k$
is equal to the cyclicity of the critical graph. Cohen et
al.~\cite{CDQV-83} seem to be the first to discover this, see also
\cite{ABG-06,BCOQ,But:10,CG:79,HOW:05}. Generalizations to reducible
case, computational complexity issues and important special cases of
the cyclicity theorem have been extensively studied in
\cite{But:10,BdS,Gav-00,Gav:04,Mol-05,Ser-11}.

Below we need this theorem in the form of CSR-representations as
in~\cite{Ser-09,SS-12}. To formulate it precisely we need the
following concepts and notation. The notions of the associated
digraph, the maximum cycle mean, the critical graph and the Kleene
star are of general importance in the tropical linear algebra.

For a matrix $A\in\Rnn$, define the associated weighted digraph
$\digr(A)=(N,E)$, where $N=\{1,\ldots,n\}$, $E=N\times N$ and each
edge $(i,j)\in E$ has weight $a_{ij}$. Conversely, each weighted
digraph with real entries on $n$ nodes corresponds to an $n\times n$
real matrix.

Let $\rad(A)$ denote the {\bf maximum cycle mean} of $A$, i.e.,
\begin{equation}
\label{mcmdef}
\rad(A)=\max\limits_{k=1}^n\max\limits_{i_1,\ldots,i_k} \frac{a_{i_1i_2}+\ldots+a_{i_ki_1}}{k}.
\end{equation}
The cycles $(i_1,\ldots,i_k)$ where the maximum cycle mean is
attained are called {\bf critical}. Further, the {\bf critical
graph}, denoted by $\crit(A)$, consists of all nodes and edges
belonging to critical cycles. As it will be emphasized later,
$\rad(A)$ also plays the role of the (unique) tropical eigenvalue of
$A$.

The sum of formal series
\begin{equation}
\label{kls-def}
A^*:= I\tplus A\tplus A^2\tplus\ldots
\end{equation}
is called the {\bf Kleene star} of $A\in\Rnn$. Here $I$ denotes the tropical identity matrix, i.e., the matrix with diagonal entries
equal to $0$ and the off-diagonal entries equal to $-\infty$. Series~\eqref{kls-def} converges if and only if
$\rad(A)\leq 0$, in which case $A^*=I\tplus A\tplus\ldots\tplus A^{n-1}$. The Kleene star satisfies
\begin{equation}
\label{kls-prop}
A^*=AA^*\tplus I.
\end{equation}

Defining the {\bf additive weight} of a walk $P$ on $\digr(A)$ as
sum of the weights of all edges contained in the walk, observe that
the following {\bf optimal walk interpretation} of tropical matrix
powers $A^k$ and Kleene star $A^*$: 1) for each pair $i,j$, the
$(i,j)$ entry of $A^k$ is equal to the greatest additive weight of a
walk connecting $i$ to $j$ with length $k$, 2) for each pair $i,j$
with $i\neq j$, the $(i,j)$ entry of $A^*$ is equal to the greatest
additive weight of a walk connecting $i$ to $j$ (with no length
restriction)\footnote{In what follows, by the weight of a walk we
mean this additive weight, and not the {\L}ukasiewicz weight defined
in Introduction}.

We introduce the notation related to $CSR$-representation. Let
$C\in\R^{n\times c}$ and $R\in\R^{c\times n}$ be matrices extracted
from the critical columns (resp. rows) of
$((A-\rad(A))^{\gamma})^*$, where $\gamma$ is the {\bf cyclicity} of
$\crit(A)$. To calculate $\gamma$ by definition, one needs to take
the g.c.d. of the lengths of all simple cycles in each strongly
connected component of $\crit(A)$, and then to take the l.c.m. of
these. Without loss of generality we are assuming that $\crit(A)$
occupies the first $c$ nodes of the associated graph.

Let $S$ be defined by
\begin{equation}
s_{ij}=
\begin{cases}
a_{ij}-\rad(A), & \text{if $(i,j)\in\crit(A)$},\\
-\infty, & \text{otherwise}.
\end{cases}
\end{equation}

By $A_{\cdot i}$ and $A_{i\cdot}$ we denote the $i$th column,
respectively the $i$th row of $A$.

\begin{theorem}
\label{CSR}
For any matrix
 $A\in\R^{n\times n}$ there exists a number $T(A)$ such that for all $t\geq T(A)$
\begin{equation}
A^t=\rad^t(A)CS^tR.
\end{equation}
Moreover, $(A^t)_{i\cdot}=\rad^t(A)S^tR_{i\cdot}$ and $(A^t)_{\cdot i}=\rad^t(A)C_{\cdot i}S^t$ for all
$i=1,\ldots,c$
\end{theorem}

This $CSR$ form of the Cyclicity Theorem~\cite{CDQV-83} was obtained
in~\cite{Ser-09}, see also\cite{SS-12}.

\subsection{Eigenproblem and Bellman equation}

Let $A\in\R^{n\times n}$. Vector $x\in\R^n$ is called a tropical
eigenvector of $A$ associated with $\lambda$ if it satisfies
$A\otimes x=\lambda\otimes x$ for some $\lambda$.
See~\cite{ABG-06,BCOQ,But:10,CG:79,GM:08,HOW:05} for general
references.

\begin{theorem}
\label{vorobyev}
Let $A\in\Rnn$, then the maximum cycle mean $\rho(A)$ is the unique tropical eigenvalue of $A$.
\end{theorem}

The set of all eigenvectors of $A$ with eigenvalue $\rho(A)$ is
denoted by $V(A,\rho)$ and called the {\bf eigencone} of $A$. It can
be described in terms of the critical graph, by the following
procedure. The critical graph of $\aminre$, the matrix with entries
$a_{ij}-\rho(A)$, consists of several strongly connected components,
isolated from each other.  In each of the components, one selects an
arbitrary index $i$ and picks the column $(\aminre)^*_{\cdot i}$ of
the Kleene star. It can be shown that for any other choice of an
index $i'$ in the same strongly connected component, the column will
be ``proportional'', i.e. $(\aminre)^*_{\cdot
i'}=\alpha\otimes(\aminre)^*_{\cdot i}=\alpha+\aminre)^*_{\cdot i}$.
In what follows, by $\Tilde{N}_C(A)$ we denote an index set
containing exactly one index from each strongly connected component
of $\crit(A)$. The following (standard) description of $V(A,\rho)$
is standard, see Krivulin~\cite{Kri-06,Kri:09}.

\if{
\begin{theorem}
\label{t:fundvecs}
Vectors $v^{(1)},\ldots,v^{(l)}$ are eigenvectors of $A$. Moreover,
\begin{equation}
\label{e:fundvecs}
V(A)=\{\bigtplus_{i=1}^l \kappa_i\otimes v^{(i)}\},
\end{equation}
and vectors $v^{(1)},\ldots,v^{(l)}$ are independent, i.e., none of
them can be expressed as a max combination of the remaining ones.
\end{theorem}

Due to this result, vectors $v^{(1)},\ldots, v^{(l)}$ are called the {\bf fundamental eigenvectors} of $A$.

In particular, the eigencone of $A$ admits the following algebraic representation.
}\fi

\begin{theorem}
\label{t:closed}
Let $A\in\R^{n\times n}$. Then
$V(A,\rho)$ consists of all vectors $(\aminre)^* z$, where $z$ is any vector in $\Rmax^n$
satisfying
\begin{equation}
\label{zdef0}
 i\notin \Tilde{N}_C(A,\lambda)\Rightarrow z_i=-\infty
\end{equation}
\end{theorem}

The tropical spectral theory can be further applied to the
equation
\begin{equation}
\label{e:bellman}
x=Ax\tplus b,
\end{equation}
which has been studied, e.g., in~\cite{Car-71}, \cite{BCOQ},
\cite{LM-98}. It is called Bellman equation due to its relations
with dynamic optimization on graphs and in particular, the Bellman
optimality principle~\cite{LM-98}.  Its nonnegative analogue is
known as $Z$-matrix equation, see~\cite{BSS-12}.

We will make use of the following basic result, formulated only recently in~\cite{BSS-12,Kri-06,Kri:09}.
The proof is given for the reader's convenience.
We consider only the case when $A$ and
$b$ have real entries, since this is the only case that we will encounter.
The solution set of~\eqref{e:bellman} will be denoted by $S(A,b)$.

\begin{theorem}
\label{schneider}
Let $A\in\R^{n\times n}$ and $b\in\R^n$. Equation~\eqref{e:bellman}
has nontrivial solutions if and only if $\rad(A)\leq 0$. In this case,
\begin{equation}
\label{sab-descr}
S(A,b) = \{A^*b\tplus v\colon Av=v\}
\end{equation}
In particular, $A^*b$ is the only solution if $\rad(A)<0$.
\end{theorem}
\begin{proof} First it can be verified that any vector like on the r.h.s.
of~\eqref{sab-descr} satisfies $x=Ax\tplus b$, using that
$A^*=AA^*\tplus I$ \eqref{kls-prop}.

Iterating the equation $x=Ax\tplus b$ we obtain
\begin{equation}
\begin{split}
x=Ax\tplus b= &A(Ax\tplus b)\tplus b=\ldots\\
=& A^k x \tplus (A^{k-1}\tplus\ldots\tplus I)b= A^k x\tplus A^*b
\end{split}
\end{equation}
for all $k\geq n$.

This implies $x\geq A^*b$. In particular, a solution of~\eqref{e:bellman} exists if and only if
$A^*$ converges, that is, if and only if $\rad(A)\leq 0$.

Further, $x$ satisfies $Ax\leq x$, hence
$x\geq Ax\geq\ldots\geq A^kx\geq\ldots$.

If $\rad(A)<0$, then by the cyclicity theorem, vectors of
$\{A^kx\}_{k\geq 1}$ start to fall with the constant rate $\rad(A)$,
starting from some $k$. This shows that for large enough $k$,
$A^kx\leq A^*b$ and $A^*b$ is the only solution.

If $\rad(A)=0$, then the orbit $\{A^kx\}_{k\geq 1}$ starts to cycle
from some $k$. But as $Ax\leq x$, we have $x\geq Ax\geq\ldots\geq
A^kx\geq\ldots$, and it is only possible that the sequence
$\{A^kx\}_{k\geq 1}$ stabilizes starting from some $k$. That is,
starting from some $k$, vector $v=A^k x$ satisfies $Av=v$. The proof
is complete.
\end{proof}

We will rather need the following formulation of the above result, implied by
Theorem~\ref{t:closed}.

\begin{corollary}
\label{c:bellclosed}
Let $A\in\R^{n\times n}$. Vector $x$ solves $x=Ax\oplus b$ if and only if it can be written
$x=A^*(b\oplus z(x))$ where $z(x)$ is a vector such that
\begin{equation}
\label{zdef00}
i\notin \Tilde{N}_C(A,0)\Rightarrow z(x)_i=-\infty.
\end{equation}
\end{corollary}

\section{Max-{\L}ukasiewicz eigenproblem}

The problem is to find $\lambda$ such that there exist nonzero $x$ solving
$A\lmult x=\lambda\lmult x$. Using~\eqref{convert} we convert
this problem to the following one:
\begin{equation}
\label{mp-spectral}
\amine x\tplus\bzero =(\lambda-1)x\tplus\bzero,\quad 0\leq x_i\leq 1.
\end{equation}

Before developing any general theory, let us look at some
two-dimensional examples. Take
\begin{equation}
A=
\begin{pmatrix}
0.5 & 0.25\\
0.25 & 0.5
\end{pmatrix}
\end{equation}
Then $A\lmult x=\lambda\lmult x$ is equivalent to
\begin{equation}
\label{ex2:luk}
\begin{split}
\max(-0.5+x_1,\; -0.75+x_2,\; 0)&=\max(\lambda-1+x_1,\; 0)\\
\max(-0.75+x_1,\; -0.5+x_2,\; 0)&=\max(\lambda-1+x_2,\; 0)
\end{split}
\end{equation}

Take $\lambda<0.5$. Then the terms with $\lambda$ can be cancelled,
and we obtain the system of inequalities
\begin{equation}
\begin{split}
& \max(-0.5+x_1,\; -0.75+x_2)\leq 0\\
& \max(-0.75+x_1,\; -0.5+x_2)\leq 0,
\end{split}
\end{equation}
which has solution set
\begin{equation}
\label{e:solution1}
\bigl\{x\colon x_1\leq 0.5\ \wedge\  x_2\leq 0.5\bigr\}.
\end{equation}
This is the {\L}ukasiewicz eigenspace\footnote{In what follows, we
omit the prefix ``max-'' by the abuse of language.} associated with
any $\lambda<0.5$.

If $\lambda>0.5$, then the diagonal terms on the l.h.s.
of~\eqref{ex2:luk} can be cancelled and we obtain the system
\begin{equation}
\label{lambdabig}
\begin{split}
\max(-0.75+x_2,\; 0)&=\max(\lambda-1+x_1,\; 0)\\
\max(-0.75+x_1,\; 0)&=\max(\lambda-1+x_2,\; 0)
\end{split}
\end{equation}

The solutions to equations of~\eqref{lambdabig} can be written as
\begin{equation}
\begin{split}
\bigl\{x\colon &(x_2\leq 0.75\ \wedge\ x_1\leq 1-\lambda)\ \vee \\
              &(x_2\geq 0.75\ \wedge\ x_1\geq 1-\lambda\ \wedge\  x_2=\lambda-0.25+x_1)\;\bigr\}
\end{split}
\end{equation}
and, respectively,
\begin{equation}
\bigl\{x\colon (x_1\leq 0.75\  \wedge\  x_2\leq 1-\lambda)\  \vee\\
(x_1\geq 0.75\  \wedge\  x_2\geq 1-\lambda\  \wedge\  x_2=0.25-\lambda+x_1)\;\bigr\}
\end{equation}
Intersecting these sets we obtain the {\L}ukasiewicz eigenspace
\begin{equation}
\label{e:solution2}
\bigl\{x\colon x_1\leq 1-\lambda\ \wedge\  x_2\leq 1-\lambda\bigr\}.
\end{equation}

If $\lambda=0.5$ then we solve
\begin{equation}
\begin{split}
\max(-0.5+x_1,\;-0.75+x_2,\; 0)&=\max(-0.5+x_1,\; 0)\\
\max(-0.75+x_1,\;-0.5+x_2,\; 0)&=\max(-0.5+x_2,\; 0)\\
\end{split}
\end{equation}
which is equivalent to
\begin{equation}
\begin{split}
-0.75+x_2  &\leq\max(-0.5+x_1,\; 0)\\
-0.75+x_1  &\leq\max(-0.5+x_2,\; 0).
\end{split}
\end{equation}
The solution set is
\begin{equation}
\label{e:solution3}
\bigl\{x\colon (x_2\leq 0.25+x_1\ \vee\  x_2\leq 0.75)\  \wedge\
(x_2\geq -0.25+x_1\ \vee\  x_1\leq 0.75)\; \bigl\}.
\end{equation}

{\L}ukasiewicz eigenspaces~\eqref{e:solution1},~\eqref{e:solution2} and~\eqref{e:solution3} for $\lambda=0.1$, $\lambda= 0.5$ and $\lambda=0.8$ are
displayed on Figure~\ref{f:luk-eigs1}. 

\begin{figure}
 \begin{center}
\hspace{-10pt}
  \includegraphics[width=105pt]{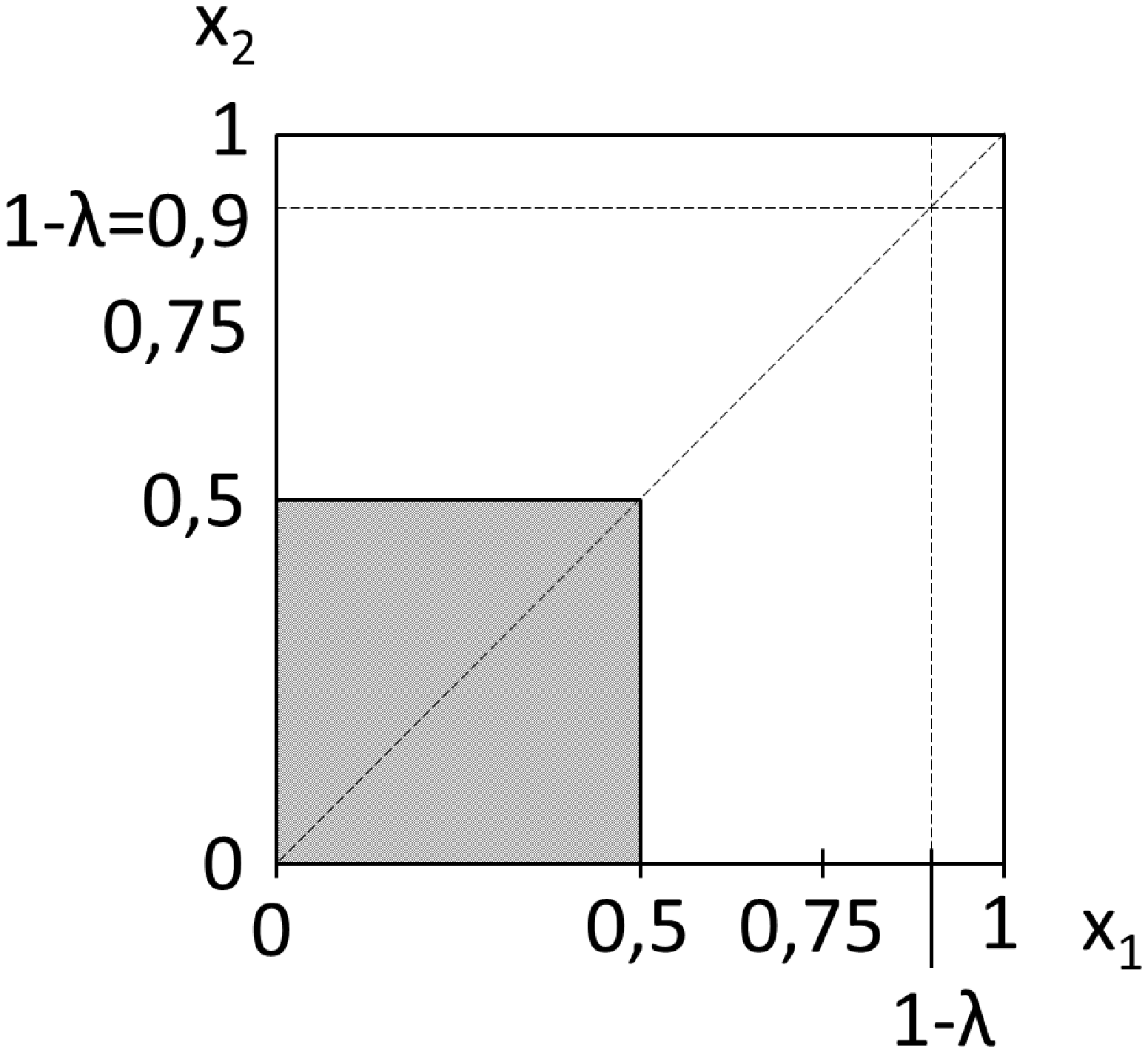}\hspace{15pt}
  \includegraphics[width=105pt]{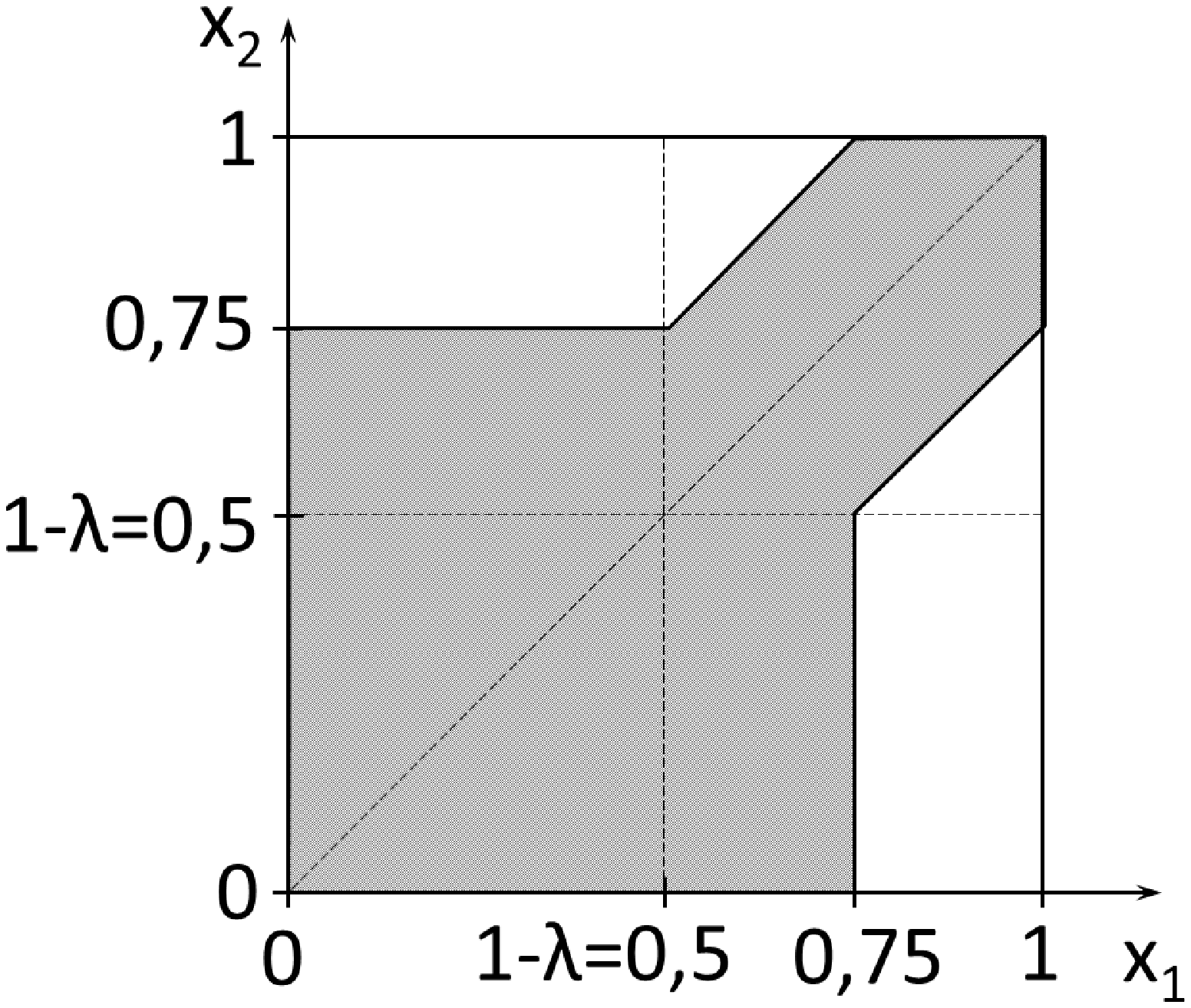}\hspace{15pt}
 \includegraphics[width=105pt]{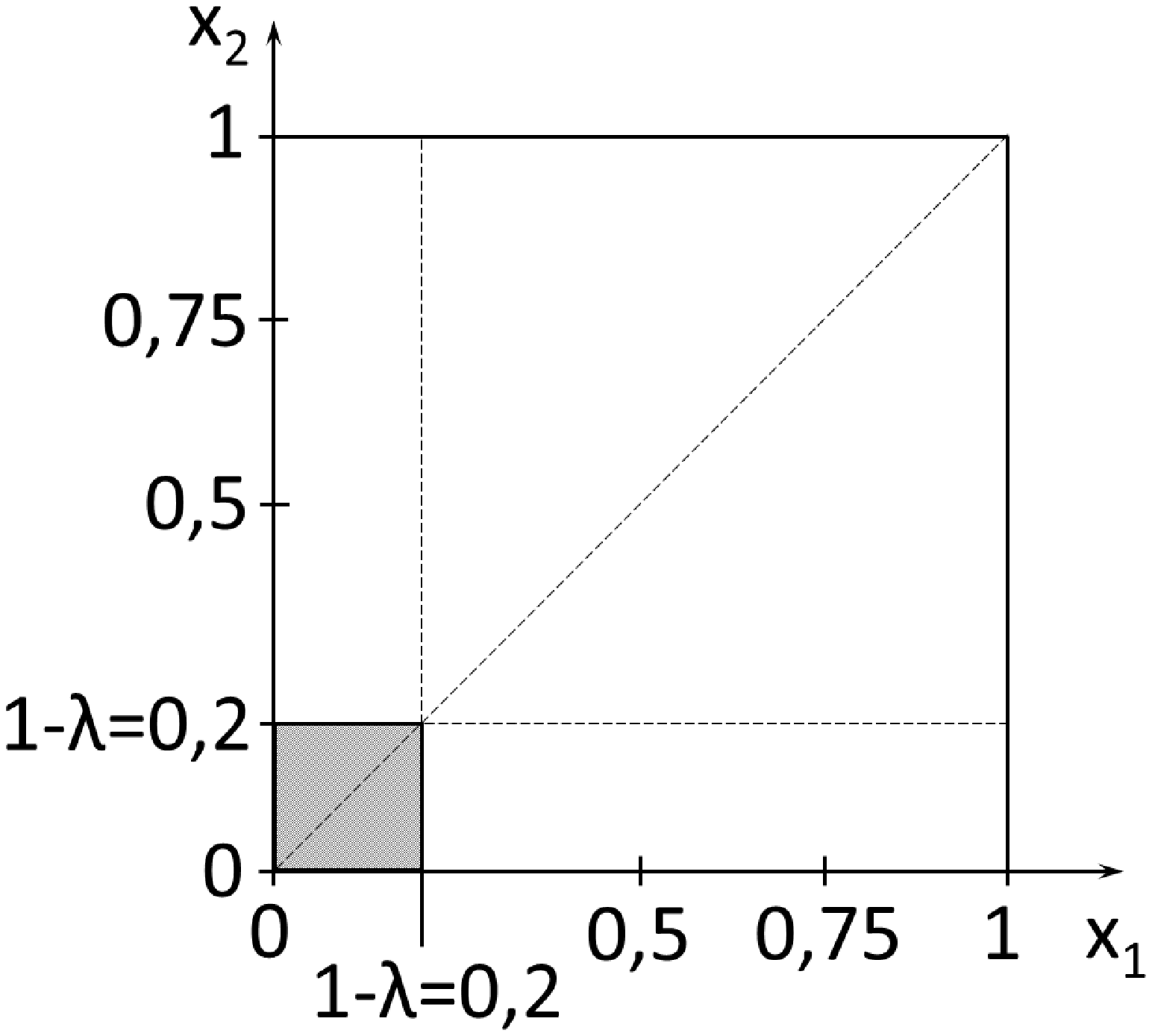}\\
  \caption{Cases $\lambda=0.1$, $\lambda= 0.5$ and $\lambda=0.8$}\label{f:luk-eigs1}
  \end{center}
\end{figure}


The relatively simple structure of the eigenspace in the previous example is influenced by the symmetry of the matrix $A$, which reduces the number of conditions that are to be considered.
Now we show another two-dimensional example, in which all entries of the matrix have different values
\begin{equation}
A=
\begin{pmatrix}
0.2 & 0.1\\
0.7 & 0.4
\end{pmatrix}
\end{equation}
There are five possible positions of the parameter $\lambda$ with respect to the diagonal entries $0.2$ and $0.4$, namely: $\lambda < 0.2$, $\lambda = 0.2$,  $0.2< \lambda  < 0.4$, $\lambda = 0.4$ and $\lambda > 0.4$.

For $\lambda < 0.4$,  i.e. in the first three cases, the solution set is
\begin{equation}
\label{e:solution123}
\bigl\{x\colon x_1\leq 0.3\  \wedge\  x_2\leq 0.6 \bigr\}\enspace.
\end{equation}
In fact, it is the solution set for the second equation, whereas solution sets for the first equation in cases $\lambda < 0.2$, $\lambda = 0.2$ and $0.2 < \lambda < 0.4$ are:
$\{x\colon x_1\leq 0.8\  \wedge\  x_2\leq 0.9 \}$,
$\{x\colon (x_1\leq 0.8 \  \wedge\  x_2\leq 0.9)\  \vee\
(x_1> 0.8 \  \wedge\  -0.1+x_2\leq x_1) \}$ and
$\{x\colon (x_1\leq 1-\lambda \  \wedge\  x_2\leq 0.9)\  \vee\
(x_1> 1-\lambda \  \wedge\  0.1-\lambda+x_2=x_1) \}$, respectively.
All these sets contain the solution set~\eqref{e:solution123}, hence the intersection of the solutions for both equations is described by~\eqref{e:solution123}.

For $\lambda = 0.4$ we obtain the solution sets for the first and the second equation
\begin{equation}
\begin{split}
&\bigl\{x\colon (x_1\leq 0.6 \  \wedge \  x_2\leq 0.9)\  \vee\
(x_1> 0.6 \  \wedge\  -0.3+x_2=x_1) \bigr \} \\
&\bigl\{x\colon (x_2\leq 0.6 \  \wedge\  x_1 \leq 0.3)\  \vee\
(x_2> 0.6 \  \wedge\  0.3+x_1 \leq x_2) \bigr\}
\end{split}
\end{equation}
and their intersection is the solution set
\begin{equation}
\label{e:solution4}
\begin{split}
\bigl\{x\colon &(x_1\leq 0.3\ \wedge\  x_2\leq 0.6)\  \vee\
(x_1 \leq 0.6\ \wedge\  0.6< x_2 \leq 0.9\ \wedge\ 0.3+x_1 \leq x_2)\ \vee\ \\
&(x_1>0.6 \ \wedge\  x_2>0.6\ \wedge\ 0.3+x_1= x_2)\bigr\}\enspace.
\end{split}
\end{equation}

For $\lambda > 0.4$ we obtain the solution sets for the equations
\begin{equation}
\begin{split}
&\bigl\{x\colon (x_1\leq 1-\lambda \  \wedge\  x_2\leq 0.9)\  \vee\
(x_1> 1-\lambda \  \wedge\  0.1-\lambda+x_2 = x_1)\bigr\} \\
&\bigl\{x\colon (x_2\leq 1-\lambda \  \wedge\  x_1 \leq 0.3)\  \vee\
(x_2> 1-\lambda \  \wedge\  0.7-\lambda+x_1 = x_2)\bigr\}\enspace.
\end{split}
\end{equation}
Their intersection (the solution set) for $0.4< \lambda <0.7$ has the form
\begin{equation}
\label{e:solution5a}
\begin{split}
\bigl\{x\colon &(x_1\leq 1-\lambda \  \wedge\  1- \lambda<x_2 \leq 0.9 \  \wedge\ 0.7- \lambda+x_1=x_2)\  \vee\ \\
&(x_1 \leq 0.3 \  \wedge\  x_2 \leq 1-\lambda)\bigr\}\enspace,
\end{split}
\end{equation}
while for $\lambda \geq 0.7$ the solution set is simply
\begin{equation}
\label{e:solution5b}
\bigl\{x\colon x_1 \leq 1-\lambda\  \wedge\  x_2 \leq 1- \lambda \bigr\}\enspace.
\end{equation}
All described eigenspaces
~\eqref{e:solution123},~\eqref{e:solution4},~\eqref{e:solution5a}
and~\eqref{e:solution5b} for $\lambda=0.35$, $\lambda= 0.4$,
$\lambda= 0.55$ and $\lambda=0.85$ are displayed on
Figure~\ref{f:luk-eigs2}.
\begin{figure}
 \begin{center}
\hspace{-10pt}
  \includegraphics[width=120pt]{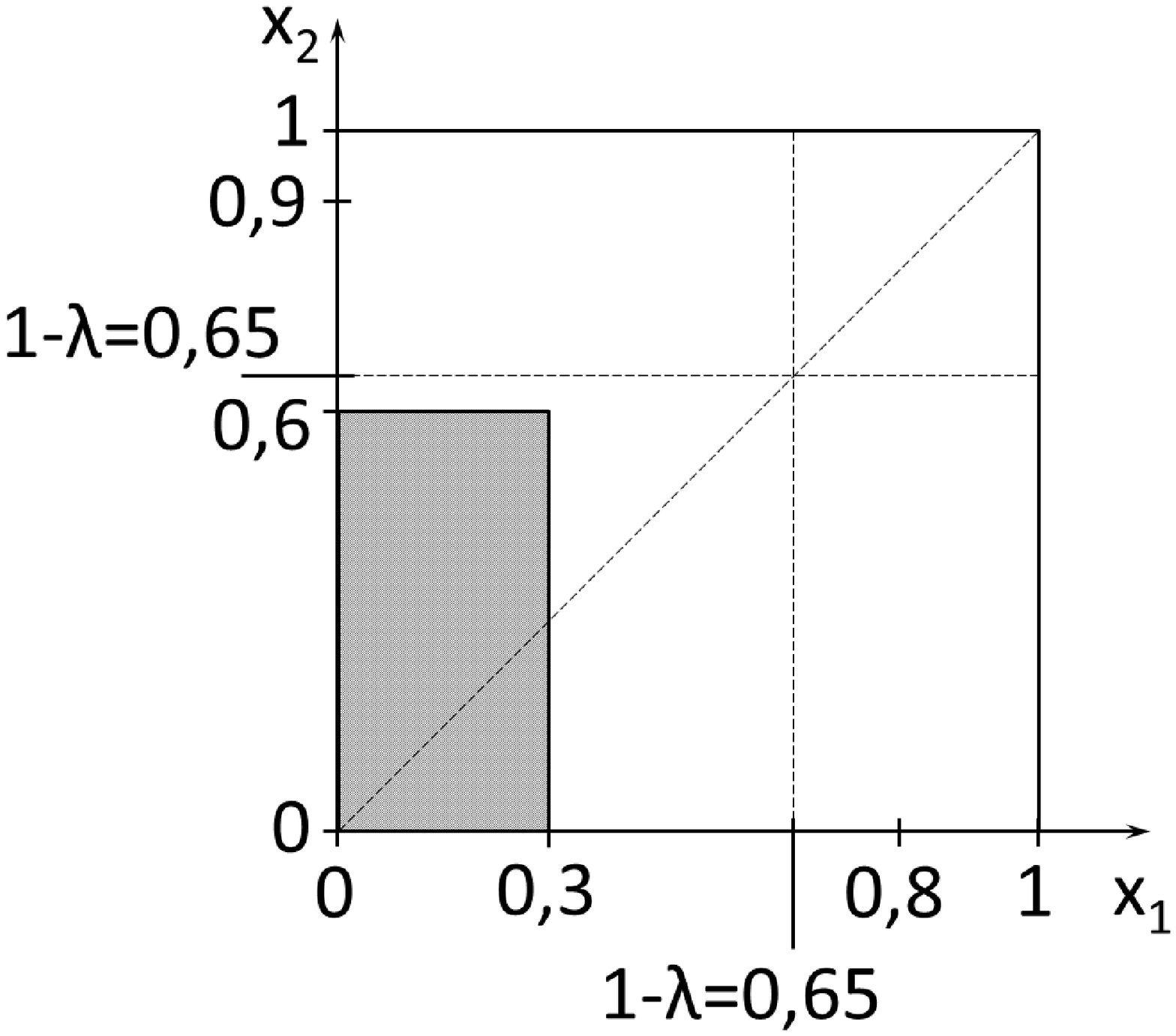}\hspace{-5pt}
  \includegraphics[width=120pt]{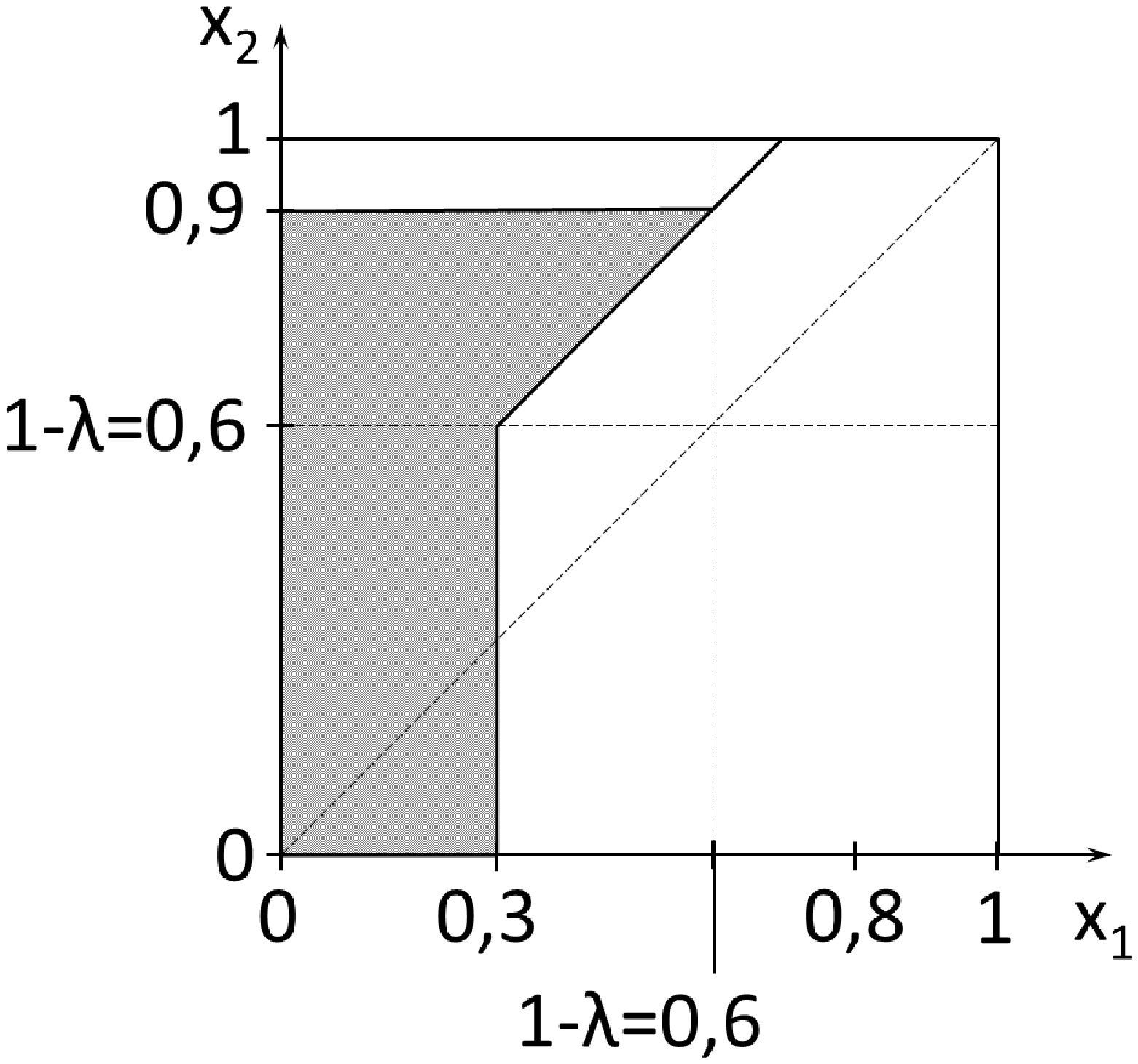}\hspace{-5pt}
  \includegraphics[width=120pt]{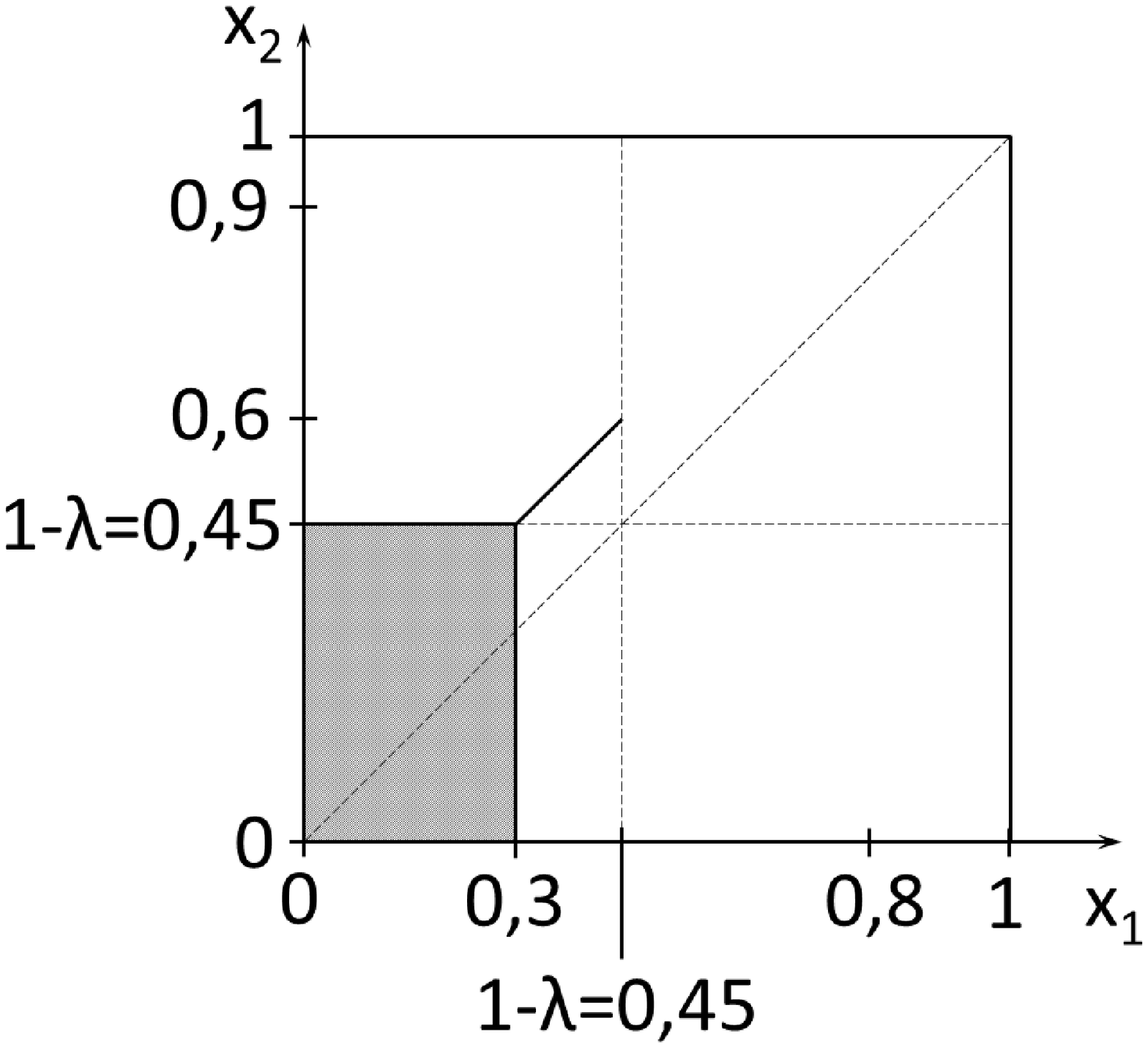}\hspace{-5pt}
 \includegraphics[width=120pt]{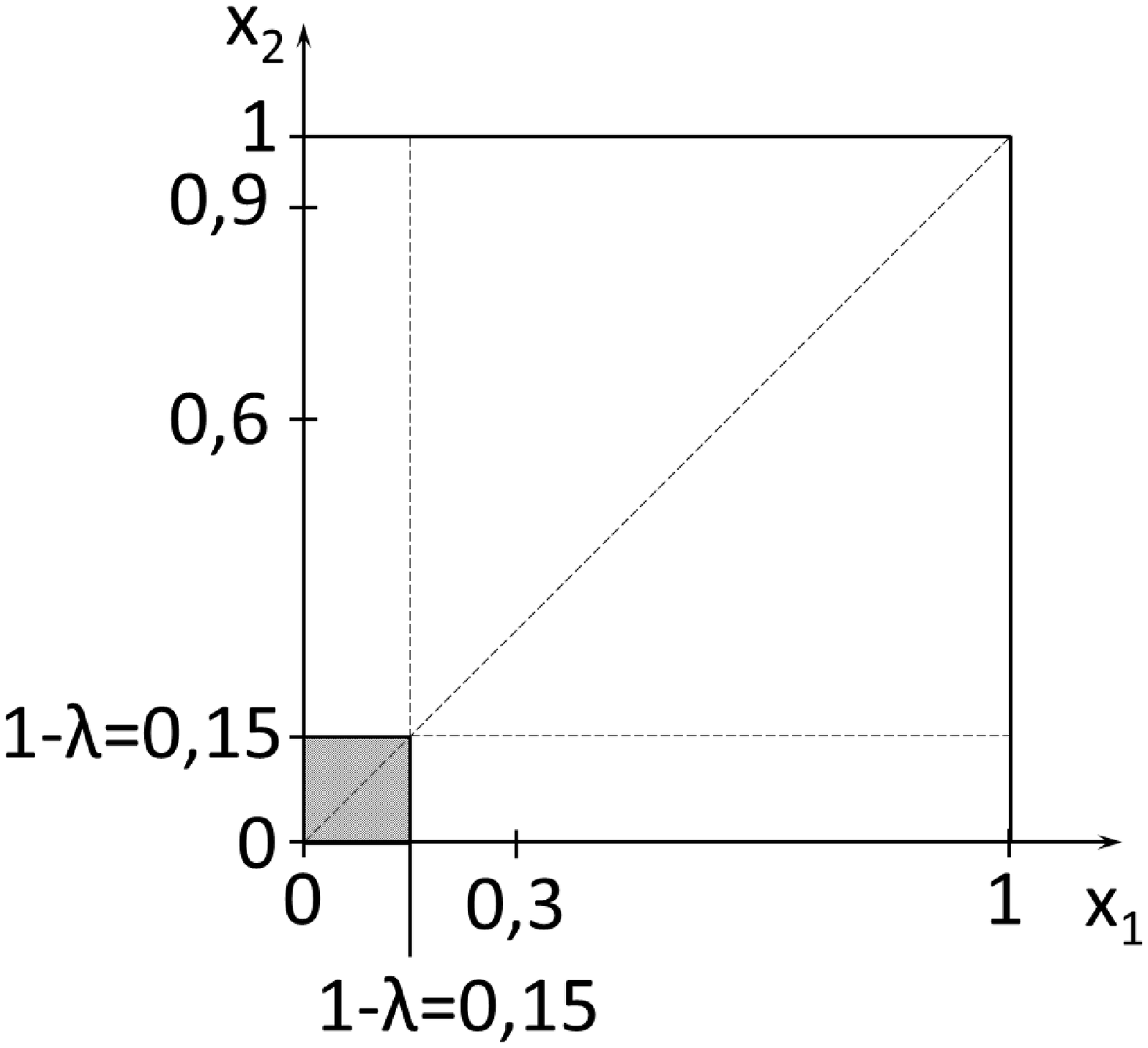}\\
  \caption{Cases $\lambda=0.35$, $\lambda= 0.4$, $\lambda= 0.55$ and $\lambda=0.85$}\label{f:luk-eigs2}
  \end{center}
\end{figure}


We now going to describe certain parts of the {\L}ukasiewicz
eigenspace, determined by the pattern of maxima on the r.h.s.
of~\eqref{mp-spectral}, whose generating sets (containing no more
than $n+1$ points) can be written explicitly and computed in
polynomial time.

\subsection{Background eigenvectors}

These are the eigenvectors that satisfy $x_i\leq 1-\lambda$ for all $i$. In this case~\eqref{mp-spectral} becomes
$\amine x\leq\bzero$, and using~\eqref{e:resid} we obtain that the solutions are given by
\begin{equation}
\label{backgr}
x_j\leq\min\left(1-\lambda,\min\limits_i (1-a_{ij})\right)
\end{equation}
If $\lambda=1$ then there are no nonzero background eigenvectors.\\
If $\lambda=0$ then all {\L}ukasiewicz eigenvectors are background.

If $\lambda<1$ then nonzero background eigenvectors exist if and
only if there is $j$ such that $\max\limits_i a_{ij}<1$. In other
words, background eigenvectors do not exist if and only if in each
column of $A$ there is an entry equal to $1$. It can be seen that in
this case $\rad(A)=1$.

In particular, we obtain the following.

\begin{proposition}
A matrix $A\in\unint^{n\times n}$ with $\rad(A)<1$ has nonzero
background {\L}ukasiewicz eigenvectors for all $\lambda<1$.
\end{proposition}

\subsection{Pure eigenvectors}

These are the vectors that satisfy $x_i\geq 1-\lambda$ for all $i$. In this case~\eqref{mp-spectral}
transforms to
\begin{equation}
\label{belleq-pure}
x= \aminle x\tplus(1-\lambda)\otimes\bzero,\quad 1-\lambda \leq x_i\leq 1.
\end{equation}
Evidently, the set of vectors satisfying~\eqref{belleq-pure} is tropically convex.

We obtain the following description of the pure eigenvectors.

\begin{theorem}
\label{t:pure}
Matrix $A\in\unint^{n\times n}$ has a pure eigenvector associated with $\lambda\in\unint$ if and only if
$\rho(A)\leq\lambda$ and $\max((\aminle)^*\bzero)\leq\lambda$. Vector $x\in\R^n$ is a pure {\L}ukasiewicz
eigenvector associated with $\lambda$ if and only if it is of the form
\begin{equation}
\label{xexpl}
x= \left(\aminle\right)^* \bigl((1-\lambda)\otimes\bzero\oplus z(x)\bigr)
\end{equation}
where $z(x)$ is a vector satisfying
\begin{equation}
\label{zconstr}
\begin{split}
z(x)\leq \left( \left(\aminle\right)^*\right)^{\sharp}\otimes'\bunity,\\
i\notin \Tilde{N}_C(A,\lambda)\Rightarrow z(x)_i=-\infty.
\end{split}
\end{equation}
In particular, $z(x)=-\infty$ whenever $\lambda\neq\rho(A)$.
\end{theorem}

\begin{proof}
As any pure eigenvector is a solution to~\eqref{belleq-pure} we can
use Corollary~\ref{c:bellclosed}. We obtain that $x\in\R^n$ is a
pure eigenvector if and only if it can be represented as
in~\eqref{xexpl} with $z(x)$ satisfying the second constraint
in~\eqref{zconstr}, and has all coordinates between $1-\lambda$ and
$1$. From~\eqref{xexpl} we conclude that $x\geq
(1-\lambda)\otimes(\aminle)^*\bzero\geq (1-\lambda)\otimes\bzero$,
so $x_i\geq 1-\lambda$ holds for all $i$. Substituting~\eqref{xexpl}
in $x\leq\bunity$ and using the duality law~\eqref{e:resid} we
obtain the condition $\max((\aminle)^*\bzero)\leq\lambda$ and the
first constraint in~\eqref{zconstr}.
\end{proof}

\begin{definition}
\label{def:lsecure}
For $\lambda>0$, a weighted graph $\digr$ is called {\em $\lambda$-secure} if
for each node of $\digr$ and each walk $P$ issuing from that node, $-\lambda+w(P)\leq 0$.
\end{definition}

\begin{corollary}
\label{c:pure-secure}
A matrix $A\in\unint^{n\times n}$ has a pure eigenvector associated with $\lambda\in\unint$ if and only if
$\digr(A^{(\lambda)})$ is $\lambda$-secure.
\end{corollary}
\begin{proof} By Theorem~\ref{t:pure}, pure
{\L}ukasiewicz eigenvector associated with $\lambda$ exists if and
only if 1)$\rho(\aminle)\leq\ 0$ and 2)$\max((\aminle)^*)\leq
\lambda$. The first condition implies that $(\aminle)^*$ exists, and
the second condition implies that  $\digr(A^{(\lambda)})$ is
$\lambda$-secure, by the walk interpretation $(\aminle)^*$.
Conversely, the $\lambda$-security is impossible if
$\digr(A^{(\lambda)})$ has a cycle with positive weight, and the
condition $\max((\aminle)^*)\leq \lambda$ follows from the optimal
walk interpretation of $(\aminle)^*$.
\end{proof}

We now describe the generating set of pure {\L}ukasiewicz eigenvectors.

\begin{corollary}
\label{c:pure} The set of pure {\L}ukasiewicz eigenvectors
associated with $\lambda\in\unint$ is the tropical convex hull of
vectors $u$ and $v^{(k)}$ for $k\in \Tilde{N}_C(A,\lambda)$ given
by~\eqref{xexpl} where
\begin{itemize}
\item[1.] $z(u)=-\infty$,
\item[2.] $z(v^{(k)})_k= (((\aminle)^*)^{\sharp}\otimes'\bunity)_k$
and $z(v^{(k)})_i=-\infty$ for $i\neq k$.
\end{itemize}
\end{corollary}
\begin{proof}
By Theorem~\ref{t:pure}, $u$ and $v^{(k)}$ are pure {\L}ukasiewicz eigenvectors associated with $\lambda$,
and so is any tropically convex combination of them. Further, all $z(x)$ satisfying~\eqref{zconstr}
are tropical convex combinations of $z(u)$ and $z(v^{(k)})$. Using max-linearity of~\eqref{xexpl},
we express any pure {\L}ukasiewicz eigenvector  as a tropical convex combination of
$u$ and $v^{(k)}$.
\end{proof}

This result also leads to a description of all {\L}ukasiewicz eigenvectors when $\lambda=1$.

\begin{corollary}
\label{c:lambdaone} When $\lambda=1$, all nontrivial {\L}ukasiewicz
eigenvectors are pure. Further, nontrivial eigenvectors exist if and
only if $\rad(A)=1$. In this case, the {\L}ukasiewicz eigenspace is
the tropical convex hull of $\bzero$ and the columns of
$1\otimes(\amine)^*$ with indices in  $\Tilde{N}_C(A,1)$ (i.e., the
fundamental eigenvectors of $A$ shifted by one).
\end{corollary}
\begin{proof}
When $\lambda=1$  we have $\lambda+x_i-1=x_i\geq 0$, hence all nontrivial eigenvectors are pure.
Next we apply Theorem~\ref{t:pure} with $\lambda=1$. Vector $(1-\lambda)\otimes (\aminle)^*\bzero$ becomes
$(\amine)^*\bzero=\bzero$, and the condition $\max(\aminle)^*\bzero\leq\lambda$ is always true.
It remains to apply Corollary~\ref{c:pure} observing that $(((\amine)^*)^{\sharp}\otimes'\bunity)_i=1$
for all $i$.
\end{proof}

Note that in this case it can be shown that $\bzero$ and $v^{(i)}$ for $i=1,\ldots,l$, are the extreme points
of the set of {\L}ukasiewicz eigenvectors.


\if{
\begin{problem}
Investigate when the pure eigenspace is convex.
\end{problem}
}\fi

\subsection{$(K,L)$ eigenvectors}

Now we consider $(K,L)$ {\L}ukasiewicz eigenvectors, i.e., $x\in[0,1]^n$ such that
$A\otimesluk x=\lambda\otimesluk x$,
$x_i\geq (1-\lambda)$ for $i\in K$ and
$x_i\leq (1-\lambda)$ for $i\in L$, where $K,L\subseteq\{1,\ldots,n\}$ are such that
$K\cup L=\{1,\ldots,n\}$ and $K\cap L=\emptyset$.

We obtain the following description of $(K,L)$-eigenvectors, which
uses the concept of tropical Schur complement, introduced by Akian,
Bapat and Gaubert~\cite{ABG-04}, Definition~2.13:
\begin{equation}
\label{def:schur}
\Schur(K,\lambda,A)=A_{LL}\tplus A_{LK} (\aminlambda_{KK})^* \aminlambda_{KL}.
\end{equation}
Observe that  $\Schur^{(1)}(K,\lambda,A)=\amine_{LL}\tplus \amine_{LK} (\aminlambda_{KK})^* \aminlambda_{KL}$.

\begin{theorem}
\label{t:KLeig}
Let $\lambda$ satisfy $0<\lambda<1$, and let $(K,L)$ be a partition of $\{1,\ldots,n\}$
where $K$ and $L$ are proper subsets of $\{1,\ldots,n\}$.
Matrix $A\in\unint^{n\times n}$ has a $(K,L)$ {\L}ukasiewicz eigenvector
associated with $\lambda$ if and only if $\rho(A_{KK})\leq\lambda$ and $\aminlambda_{LK}(\aminlambda_{KK})^*\bzero_K\leq\bzero_L$.
In this case the set of $(K,L)$-eigenvectors is given by
\begin{equation}
\label{xLconstr}
\bzero_L\leq x_L \leq \min((\Schur^{(1)}(K,\lambda,A))^{\sharp}\otimes'\bzero_L,(1-\lambda)\otimes \bzero_L),
\end{equation}
\begin{equation}
\label{xKconstr}
x_K=(\aminlambda_{KK})^*(\aminlambda_{KL}x_L\tplus (1-\lambda)\bzero_K\tplus z_K(x)),
\end{equation}
where $z_K(x)$ is any vector with components in $K$ satisfying
\begin{equation}
\label{zKconstr}
\begin{split}
z_K(x)&\leq (\amine_{LK} (\aminlambda_{KK})^*)^{\sharp}\otimes'\bzero_L,\\
i\notin N_C(A_{KK},\lambda)&\Rightarrow  z_K(x)_i=-\infty,
\end{split}
\end{equation}
and $x_K$, resp. $x_L$ are restrictions of $x$ to the index set $K$,
resp. to $L$.
\end{theorem}

\begin{proof}
By definition, $(K,L)$ {\L}ukasiewicz eigenvectors satisfy the
following system

\begin{eqnarray}
\label{Keqn}
\amine_{KL}x_L\tplus \amine_{KK}x_K\tplus\bzero_K = (\lambda-1)\otimes x_K,\\
\label{Leqn}
\amine_{LL}x_L\tplus \amine_{LK}x_K\tplus\bzero_L = \bzero_L.
\end{eqnarray}

We start by writing out the solution of~\eqref{Keqn}. Subtracting
$(\lambda-1)$ from both sides of this equation we obtain
\begin{equation}
\label{Keqn-mod}
\aminlambda_{KL}x_L\tplus \aminlambda_{KK}x_K\tplus(1-\lambda)\otimes \bzero_K = x_K
\end{equation}

Equation~\eqref{Keqn-mod} has solutions if and only if
$\rho(A_{KK})\leq\lambda$. Let $z_K$ be any $K$-vector (that is, a
vector with all components in $K$) such that
\begin{equation}
\label{zKdef}
i\notin N_C(A_{KK},\lambda)\Rightarrow (z_K)_i=-\infty.
\end{equation}
Then vector $x$ satisfies~\eqref{Keqn-mod} if and only if it
satisfies
\begin{equation}
\label{xk-expr}
x_K=(\aminlambda_{KK})^*(\aminlambda_{KL}x_L\tplus (1-\lambda)\bzero_K\tplus z_K),
\end{equation}
Observe that if $\rho(A_{KK})<\lambda$ then $\lambda$ is not the
(unique) eigenvalue of $A_{KK}$, and $N_C(A_{KK},\lambda)$ is empty
so that $z_K=-\infty$.

Equation~\eqref{Leqn} can be rewritten as inequality
\begin{equation}
\label{Leqn-mod}
\amine_{LL}x_L\tplus \amine_{LK}x_K\leq \bzero_L,
\end{equation}
and substituting $x_K$ from~\eqref{xk-expr} we get
\begin{equation}
\label{Leqn-mod2}
(\amine_{LL}\tplus \amine_{LK} (\aminlambda_{KK})^* \aminlambda_{KL})x_L \tplus \aminlambda_{LK}(\aminlambda_{KK})^*\bzero_K\tplus \amine_{LK} (\aminlambda_{KK})^* z_K\leq\bzero_L
\end{equation}
Expression in the brackets equals
$\Schur^{(1)}(K,\lambda,A)$. The inequality~\eqref{Leqn-mod2} is
equivalent to the following three inequalities:
\begin{equation}
\label{Leqn-mod3}
\begin{split}
&x_L\leq (\Schur^{(1)}(K,\lambda,A))^{\sharp}\otimes'\bzero_L,\quad
\aminlambda_{LK}(\aminlambda_{KK})^*\bzero_K\leq\bzero_L,\\
&z_K\leq (\amine_{LK} (\aminlambda_{KK})^*)^{\sharp}\otimes'\bzero_L
\end{split}
\end{equation}

We also must require that $(1-\lambda)\otimes\bzero_K\leq x_K\leq
\bunity_K$ and $\bzero_L\leq x_L\leq (1-\lambda)\otimes\bzero_L$.

However, from~\eqref{xk-expr} we conclude that $x_K\geq
(1-\lambda)(\aminlambda_{KK})^*\bzero_K\geq
(1-\lambda)\otimes\bzero_K$, so the inequality $x_K\geq
(1-\lambda)\otimes\bzero_K$ is automatically fulfilled.
In particular, this implies the existence of a nontrivial $(K,L)$
eigenvector associated with $\lambda<1$, provided that
$\rho(\aminlambda_{KK})\leq 0$ and
$\aminlambda_{LK}(\aminlambda_{KK})^*\bzero_K\leq\bzero_L$.

If $x_K$ and $x_L$ satisfy~\eqref{xk-expr} and~\eqref{Leqn-mod3}
then in particular we have~\eqref{Leqn-mod}, and hence
$x_K\leq(\amine)^{\sharp}_{KL}\otimes'\bzero_L\leq \bunity_K$. Thus
$(1-\lambda)\otimes\bzero_K\leq x_K\leq \bunity_K$ is satisfied
automatically, and the claim of the theorem follows.
\end{proof}

We now give a graph-theoretic interpretation for the existence conditions in
Theorem~\ref{t:KLeig}.

\begin{definition}
\label{def:secure} Let $\digr$ be a weighted digraph with nodes
$\{1,\ldots,n\}$. Let $K$ and $L$ be proper subsets of
$\{1,\ldots,n\}$ such that $(K,L)$ is a partition of
$\{1,\ldots,n\}$. This partition is called secure, with respect to
$\digr$, if every walk in $\digr$, that starts in a node $\ell\in L$
and has all other nodes in $K$, has a non-positive weight.\\

We say that the partition $([n],\emptyset)$ is secure if
$\digr(A^{(\lambda)})$ is $\lambda$-secure, and that the partition
$(\emptyset,[n])$ is secure.
\end{definition}

For example, consider the {\L}ukasiewicz eigenproblem with
$$
A=
\begin{pmatrix}
0.3 & 0.1 & 0.2 & 0 & 0.3\\
0.7 & 0.3 & 0.5 & 0.5 & 0.3\\
0.3 & 0.2 & 0.3 & 0.5 & 0.3\\
0.1 & 0.2 & 0.1 & 0.3 & 0.3\\
0.3 & 0 & 0.2 & 0.2 & 0.3
\end{pmatrix}
$$
and $\lambda=0.4$. This is the associated matrix of the graph on
Figure~\ref{f:network} (right). Consider secure partitions of the
graph associated with $A^{(\lambda)}$. One of them has $L=\{1,5\}$,
$K=\{2,3,4\}$ and is shown on Figure~\ref{f:KLex}. In the graph
associated with $A^{(\lambda)}$, it is secure to travel from any
node that belongs to $L$ to any node belonging to $K$ and to
continue a walk in $K$ without having any deficit of the budget. On
the other hand, the partition $L=\{5,4\}$, $K=\{1,2,3\}$ is not
secure, because there is a walk from index $4 \in L$ to index $1 \in
K$ with positive weight (the walk $p=(4,2,1)$ with $w(p)=0.1$).

\begin{figure}
 \begin{center}
  \includegraphics[width=200pt]{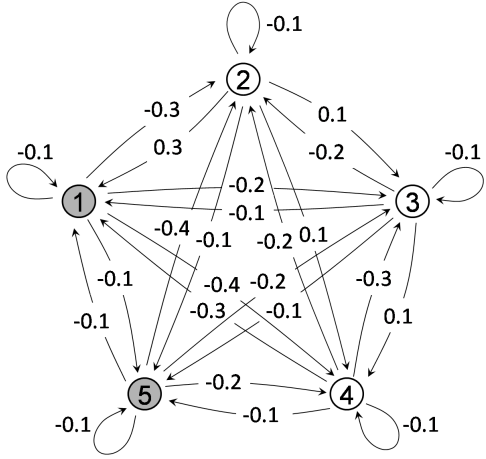}
  \caption{Secure partition $L=\{1,5\}$, $K=\{2,3,4\}$ ($\lambda=0,4$)}\label{f:KLex}
  \end{center}
\end{figure}

\begin{proposition}
\label{t:partitions}
Let $\lambda$ satisfy $0<\lambda<1$, and let $(K,L)$ be a partition of $\{1,\ldots,n\}$
where $K$ and $L$ are proper subsets of $\{1,\ldots,n\}$.
Matrix $A\in\unint^{n\times n}$ has a $(K,L)$ {\L}ukasiewicz eigenvector
associated with $\lambda$ if and only if $(K,L)$ is a secure partition.
\end{proposition}

\begin{proof}
We have the following necessary and sufficient conditions for existence of a nontrivial {\L}ukasiewicz eigenvector
with a given $\lambda$: 1)$\rho(\aminlambda_{KK})\leq 0$ and
2)$\aminlambda_{LK}(\aminlambda_{KK})^*\bzero_K\leq\bzero_L$.
Using the walk interpretation of $(\aminlambda_{KK})^*$ we see that 1) and 2) imply that $(K,L)$ is secure. Conversely, $(K,L)$ cannot be secure if $\rho(\aminlambda_{KK})> 0$. Indeed, then
there exists a cycle in $K$ with weight exceeding $0$. Any walk connecting to this cycle and following it will
have a positive weight, starting from some length, thus $\rho(\aminlambda_{KK})\leq 0$. In this case $(\aminlambda_{KK})^*$ exists and we have $\aminlambda_{LK}(\aminlambda_{KK})^*\bzero_K\leq\bzero_L$ since $(K,L)$ is secure.
\end{proof}

As a corollary of Theorem~\ref{t:KLeig}, we can describe the set of $(K,L)$-eigenvectors as a tropical convex hull of
no more than $n+1$ explicitly defined points. Both statement and proof are analogous to those of Corollary~\ref{c:pure}.

\begin{corollary}
\label{c:KLeig} The set of $(K,L)$ {\L}ukasiewicz eigenvectors is
the tropical convex hull of vectors $u$, $v^{(\ell)}$ for $\ell\in
L$, and $w^{(k)}$ for $k\in K$ such that $k\in N_C(A_{KK},\lambda)$,
defined by~\eqref{xKconstr} and the following settings:
\begin{itemize}
\item[1.] $u_L=\bzero_L$ and $z_K(u)=-\infty$,
\item[2.] $v^{(\ell)}_i=0$ for $i\in L$ and $i\neq\ell$, $v^{(\ell)}_{\ell}=\min((\Schur^{(1)}(K,\lambda,A))^{\sharp}\otimes'\bzero_L)_{\ell},1-\lambda)$ and
$z_K(v^{(\ell)})=-\infty$,
\item[3.] $w^{(k)}_L=\bzero_L$, $z_K(w^{(k)})_k=((\amine_{LK} (\aminlambda_{KK})^*)^{\sharp}\otimes'\bzero_L)_k$
and $z_K(w^{(k)})_i=-\infty$ for $i\neq k$.
\end{itemize}
and the $K$-subvectors of $u$, $v^{(l)}$ and $w^{(k)}$ are given by~\eqref{xKconstr}.
\end{corollary}
\begin{proof}
By Theorem~\ref{t:KLeig}, $u$, $v^{(\ell)}$, $w^{(k)}$
are {\L}ukasiewicz $K,L$-eigenvectors associated with $\lambda$,
and so is any tropically convex combination of them. Further, all $z_K(x)$ satisfying~\eqref{zKconstr}
are tropical convex combinations of $z_K(u)=-\infty$ and $z_K(v^{(\ell)})$. Using max-linearity of~\eqref{xKconstr},
we express any pure {\L}ukasiewicz eigenvector  as a tropical convex combination of
$u$, $v^{(\ell)}$ and $w^{(k)}$.
\end{proof}

\subsection{Finding all secure partitions}

The problem of describing all {\L}ukasiewicz eigenvectors associated
with $\lambda$ can be associated with the following combinatorial
problem: 1) check whether graph $\digr(A^{(\lambda)})$ is
$\lambda$-secure, 2) describe all $(K,L)$ partitions which are
secure in $\digr(A^{(\lambda)})$.


Node $i$ of a weighted digraph is called {\bf secure} if the weight of every walk starting in $i$
is nonpositive.

\begin{theorem}
\label{t:secure-part}
Let $(K_1,L_1)$ and $(K_2,L_2)$ be two different secure partitions such that
$L_1\subseteq L_2$.
\begin{itemize}
\item[1.] There is a node $k\in L_2\backslash L_1=K_1\backslash K_2$, secure in $\digr(A^{(\lambda)}_{K_1K_1})$.
\item[2.] Partition $(K_1-k,L_1+k)$ is secure if and only if $k$ is secure in $\digr(A^{(\lambda)}_{K_1K_1})$.
\end{itemize}
\end{theorem}
\begin{proof}
1. Consider the set $M:=K_1\backslash K_2$ and assume that no node of $M$ is secure in $K_1$. Then for each node
$i\in M$ there exists a walk $P$ in $\digr(A^{(\lambda)}_{K_1K_1})$ starting in $i$, whose weight is positive, and such that the weights of all its proper subwalks starting in $i$ are non-positive. Let $\ell$ be the end node of the walk, and let $k$ be the last node of $M$ visited by the walk. If $k\neq\ell$ then the subwalk of $P$ starting in $k$ and ending in $\ell$ has a positive
weight, which contradicts with the security of $(K_2,L_2)$. We conclude that for all walks in $\digr(A^{(\lambda)}_{K_1K_1})$
starting in $M$ and having positive weight, the end node also belongs to $M$. However, using such walks and the finiteness of $M$ we obtain
that $\digr(A^{(\lambda)}_{K_1K_1})$ has a cycle with positive weight, contradicting the
security of $(K_1,L_1)$. Hence $M$ contains a node that is secure in $K_1$.

2. The ``if'' part: Observe that all walks starting in $L_1$ and continuing in $K_1-k$ have nonpositive weight since tha partition
$(K_1,L_1)$ is secure. All walks starting in $k$ and continuing in $K_1-k$ have nonpositive weight since $k$ is secure in $\digr(A^{(\lambda)}_{K_1K_1})$.\\
The ``only if'' part: If $k$ is not secure, then there is a walk in $\digr(A^{(\lambda)}_{K_1K_1})$ going out of $k$ and having a
positive weight. If this walk ends in $k$ then we obtain a cycle with positive weight, in contradiction with the security of $(K_1,L_1)$. Otherwise, take a subwalk which contains $k$ only as the starting node, and then its weight must be also positive. This shows that $(K_1-k,L_1+k)$ is not secure.
\end{proof}

Let us emphasize that the above theorem also holds in the case of
$K_1=[n]$ when we require that $\digr(A^{(\lambda)})$ is
$\lambda$-secure, since this condition (by Theorem~\ref{t:pure} and
Corollary~\ref{c:pure}) implies that $\rho(A^{(\lambda)})\leq 0$.
The case when $L_2=[n]$, in which case $(K_2,L_2)$ is secure by
definition, yields the following corollary.

\begin{corollary}
\label{c:secure-part}
Let $(K,L)$ be a secure partition, then
\begin{itemize}
\item[1.] There is a node $k\in K$, secure in $\digr(A^{(\lambda)}_{KK})$.
\item[2.] Partition $(K-k,L+k)$ is secure if and only if $k$ is secure in $\digr(A^{(\lambda)}_{KK})$.
\end{itemize}
\end{corollary}

The equivalence of security of partition $(K,L)$ with existence of $(K,L)$-{\L}ukasiewicz
eigenvector implies that there is a unique minimal secure partition.

\begin{proposition}
\label{p:least-secure} Digraph $\digr(A^{(\lambda)})$ has the least
secure partition, which corresponds to the greatest {\L}ukasiewicz
eigenvector associated with $\lambda$.
\end{proposition}
\begin{proof}
First observe that there is the set of {\L}ukasiewicz eigenvectors contains the greatest element:
this is a tropical convex hull of a finite number of extremal points,
and the greatest element is precisely the ``sum'' ($\oplus$) of all extremal points.

Let $z$ be this greatest element, and let $(K,L)$ be the partition such that
$z_i<1-\lambda$ for all $i\in L$ and $z_i\geq 1-\lambda$ for all $i\in K$. This partition is secure.
For any other secure partition $(K',L')$ there exists a $(K',L')$-eigenvector $y$ (associated with $\lambda$),
and we have $y\leq z$ and hence $y_i<1-\lambda$ for all $i\in L$ implying $L\subseteq L'$.
\end{proof}

Theorem~\ref{t:secure-part}, Corollary~\ref{c:secure-part} and Proposition~\ref{p:least-secure} show that
all secure partitions of $\digr(A^{(\lambda)}$ can be described by means of the following algorithm:

\begin{algorithm}(Describing all secure partitions of $\digr(A^{(\lambda)})$)\\

{\bf Part 1.} For each $i\in\{1,\ldots,n\}$, solve
\begin{equation}
\label{e:tlp}
\begin{split}
& \max x_i,\quad 0\leq x_i\leq 1\;\text{s.t.} \; \forall i,\\
& A^{(1)}\otimes x\oplus \bzero =(\lambda-1)\otimes x\oplus \bzero
\end{split}
\end{equation}
The vector $\overline{x}$ whose components are solutions of~\eqref{e:tlp} is the greatest {\L}ukasiewicz
eigenvector associated with $\lambda$. The partition $(K,L)$ such that $\overline{x}_i<1-\lambda$ for all
$i\in L$ and $\overline{x}_k\geq 1-\lambda$ for $k\in K$ is secure partition
with the least $L$.

Problem~\eqref{e:tlp} is an instance of the tropical linear programming problems
treated by Butkovic and Aminu~\cite{BA-08,But:10} and Gaubert~et~al.~\cite{GKS-12}.
It can be solved in pseudopolynomial time~\cite{BA-08,But:10}, by means of a psudopolynomial
number of calls to a mean-payoff game oracle~\cite{GKS-12}.

{\bf Part 2.} For each secure partition $(K,L)$, all secure
partitions of the form $(K-k, L+k)$ can be found by means of
shortest path algorithms with complexity at most $O(n^3)$ (for the
classical Floyd-Warshall algorithm). All secure partitions can be
then identified by means of a Depth First Search procedure.
\end{algorithm}
\begin{proof}
Part 1. is based on Proposition~\ref{p:least-secure}.

For part 2., note that by Corollary~\ref{c:secure-part}, for each secure $(K,L)$ with $K$ nonempty, there exist greater secure partitions.  By Theorem~\ref{t:secure-part} part 2, for each secure partition $(K,L)$ and for each $k\in K$, partition $(K-k, L+k)$ is secure if and only if $k$ is secure $\digr(A_{KK}^{(\lambda)})$. This holds if and only if the $k$th row of $(A_{KK}^{(\lambda)})^*$ has all components nonpositive.
Matrix $(A_{KK}^{(\lambda)})^*$ can be computed in $O(n^3)$ operations.
\end{proof}

{\bf Example}

The example shows the work of the algorithm - describing all secure
partitions of $\digr(A^{(\lambda)})$.

Let $\lambda= 0.6$ and consider
\begin{equation}
A=
\begin{pmatrix}
0.6 & 0.7 & 0.2 \\
0.4 & 0.5 & 0.7 \\
0.3 & 0.2 & 0.4
\end{pmatrix}.
\end{equation}

Then we have

\begin{equation}
A^{(1)}=
\begin{pmatrix}
-0.4 & -0.3 & -0.8 \\
-0.6 & -0.5 & -0.3 \\
-0.7 & -0.8 & -0.6
\end{pmatrix},\quad
A^{(0.6)}=
\begin{pmatrix}
0 & 0.1 & -0.4 \\
-0.2 & -0.1 & 0.1 \\
-0.3 & -0.4 & -0.2
\end{pmatrix}.
\end{equation}

The {\L}ukasiewicz eigenproblem with $\lambda=0.6$ amounts to
solving $A^{(1)}x\oplus\bzero=(-0.4+x)\oplus\bzero$ subject to
$0\leq x_i\leq 1$ for all $i$.

According to the algorithm, we first need to find the greatest
max-{\L}ukasiewicz eigenvector. It can be verified that $(1\; 0.8\;
0.7)$ is a pure {\L}ukasiewicz eigenvector (note that all components
are greater than $0.4$). As there is a pure eigenvector, it follows
that the greatest max-{\L}ukasiewicz eigenvector is also pure, and
the minimal partition is the least partition $K=\{1,2,3\}$,
$L=\{\emptyset\}$. In fact, $(1\; 0.8\; 0.7)$ is the greatest
solution of the tropical Z-matrix equation $x=A^{(0.6)}x\oplus
(0.4+\bzero)$ satisfying $x_i\leq 1$, hence this is exactly the
greatest (pure) max-{\L}ukasiewicz eigenvector. Moreover, it can be
found that the pure eigenspace is the max-plus segment with extremal
points $(1\; 0.8\; 0.7)$ and $(0.6\; 0.5\; 0.4)$.

The second part of the algorithm describes how to find all secure
partitions. The algorithm starts with the secure partition with
minimal $L$, according to the ~\eqref{c:secure-part}, and tries to
increase $L$ by adding indices from $K$. In the beginning we have
$A^{(\lambda)}_{KK}= A^{(0.6)}$ written above.

To add an index $i \in K$ to $L$, we have to check, whether $i$ is
secure in $\digr(A^{(\lambda)}_{KK})$, i.e., whether the weight of
every walk starting in $i$ is nonpositive in
$\digr(A^{(\lambda)}_{KK})$. Index $1$ is not secure, because the
walk $p=(1,2)$ has positive weight $w(p)=0.1$, that is, $1$ can not
be added to $L$. Similarly, $2$ is not secure and can not be added
to $L$. Index $3$ is secure, because any walk starting in $3$ has
non positive weight. Hence the next secure partition is $K=\{1,2\}$
and $L=\{3\}$. We find that
\begin{equation}
A_{KK}^{(0.6)}=\left(
 \begin{array}{r r}
  0  &  0.1   \\
-0.2 & -0.1
\end{array}
\right),\quad
(A_{KK}^{(0.6)})^*=\left(
 \begin{array}{r r}
  0  &  0.1   \\
-0.2 & 0
\end{array}
\right).
\end{equation}
Using Theorem~\ref{t:KLeig} and Corollary~\ref{c:KLeig} we find that
the $(K,L)$ eigenspace is the tropical convex hull of  
$u=(0.5\; 0.4\; 0)$, $v^{(3)}=(0.6\; 0.5\; 0.4)$ and $w^{(1)}=(0.7\; 0.5\;
0)$. (Note that $v^{(3)}$ is also a pure eigenvector. In general, some vectors
can be considered as $(K,L)$ eigenvectors for several choices of $(K,L)$.)

Similarly as above, index $1$ is not secure in
$\digr(A_{KK}^{(0.6)})$, because of the walk $p=(1,2)$ with
$w(p)=0.1$. On the other hand, $2$ is secure in
$\digr(A_{KK}^{(0.6)})$. Thus index $2$ can be added and we get
further secure partition $(K,L)$ with $K=\{1\}$ and $L=\{2,3\}$.
Then $A_{KK}^{(0.6)}=(A_{KK}^{(0.6)})^*$ consists just of one entry
$0$. Using Theorem~\ref{t:KLeig} and Corollary~\ref{c:KLeig} we find
that the $(K,L)$ eigenspace is the tropical convex hull of
$u=(0.4\; 0\; 0)$, $v^{(2)}=(0.5\; 0.4\; 0)$,
$v^{(3)}=(0.4\; 0\; 0.3)$ and $w^{(1)}=(0.6\; 0\; 0)$.
\if{
\begin{equation}
A_{KK}^{(0.6)}=\left(
 \begin{array}{r}
  0
\end{array}
\right),
\end{equation}
}\fi

Index $1 \in K$ is secure in $\digr(A_{KK}^{(0.6)})$ and we get the
last secure partition $(K, L)$ with $L=\{1,2,3\}$, giving the
background eigenvectors.  See figure ~\ref{f:sec-part}). The
greatest background eigenvector is $(0.4\; 0.3\; 0.3)$.

\begin{figure}
 \begin{center}
  \includegraphics[width=140pt]{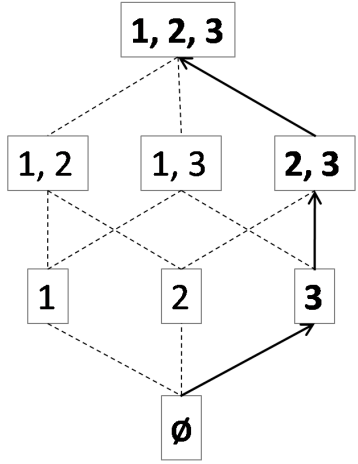}\hspace{-5pt}
  \caption{Secure partition of $\digr(A^{(\lambda)}$ (only sets $L$ are shown)}\label{f:sec-part}
  \end{center}
\end{figure}


\section{Powers and orbits}

As we are going to show, any vector orbit in max-{\L}ukasiewicz algebra is ultimately periodic.
This ultimate periodicity can be described in terms of CSR representation of periodic powers
$(\amine)^{t}$. We assume the critical graph of $A$ occupies the first $c$ nodes. Note that the
CSR terms computed for $A^t$ and $(\amine)^t$ are the same.

Observe that if $\rad(A)=1$ then
the critical graph $\crit(\amine)$ consists of all cycles, such that all entries of
these cycles have zero weight.  Using Theorem~\ref{schneider} we obtain the following.

\begin{theorem}
\label{CSR-luk}
For any $A\in\unint^{n\times n}$ and $x\in\unint^{n}$ there exists a number $T(A)$ such that for all $t\geq T(A)$,
\begin{itemize}
\item[1.] $A^{\lmult t}\lmult x =\bzero$ if $\rad(A)<1$,
\item[2.] $A^{\lmult t}\lmult x = CS^tRx\lplus\bzero$ if $\rad(A)=1$.
In this case $(A^{\otimesluk t}\otimesluk x)_i=S^tR(x\oplus\bzero)_i$ for all
$i=1,\ldots, c$
\end{itemize}
\end{theorem}
\begin{proof}
We start by showing the following identity:
\begin{equation}
\label{convert-vecs}
A^{\lmult t} \lmult x = (\amine)^t x\tplus\bzero,
\end{equation}
Indeed, iterating~\eqref{convert} we obtain
\begin{equation*}
\begin{split}
A^{\lmult t}\lmult x &= (\amine) (A^{\lmult t-1}\lmult x)\tplus\bzero \\
&=\ldots\\
&=(\amine)^t x\tplus ((\amine)^{(t-1)}\bzero \tplus \ldots \tplus (\amine)\bzero \tplus \bzero)\\
&=(\amine)^t x\tplus\bzero.
\end{split}
\end{equation*}
For the final reduction we used that all entries of $\amine$ are nonpositive, hence $(\amine)^l\tmult\bzero\leq\bzero$ for all
$l\geq 0$. Now we substitute the result of Theorem~\ref{CSR} observing that $\rad^t(A)CS^tRx\leq\bzero$ for all
large enough $t$ if $\rad(A)<1$.

If $\rad(A)=1$ then for each $t$ and $i=1,\ldots, c$, the $i$th row of $(\amine)^t$ has a zero entry, from
which it follows that $((\amine)^t\otimes\bzero))_i=0$ for all large enough $t$ and $i=1,\ldots,c$ implying that
$(\amine)^t\otimes x\oplus\bzero)_i$ equals $(\amine)^t(x\oplus\bzero)_i$ and hence
$S^tR(x\oplus\bzero)_i$ for all big enough $t$.
\end{proof}

It also follows that
\begin{equation}
\label{convert-mats}
A^{\lmult t} = (\amine)^{t-1}A\tplus\bzero,
\end{equation}
where $\bzero$ is the $n\times n$ matrix consisting of all zeros. Consequently, we obtain that
if $\rad(A)<1$ then $A^{\lmult t}=\bzero$ for sufficiently large $t$, and if not then
\begin{equation}
\label{CSR-mats}
A^{\lmult t} =CS^{t-1}RA\tplus\bzero
\end{equation}

For any $n\times n$ matrix $A$ in max-{\L}ukasiewicz algebra, we can define the ``matrix of ones''
$A^{[1]}$ by
\begin{equation}
a^{[1]}_{ij}=
\begin{cases}
1, & \text{if $a_{ij}=1$},\\
0, & \text{otherwise}.
\end{cases}
\end{equation}

\begin{theorem}
Orbits of vectors and matrix powers of a matrix $A$ in max-{\L}ukasiewicz algebra are ultimately periodic.
They are ultimately zero if $\rad(A)<1$. Otherwise if $\rad(A) =1$,
\begin{itemize}
\item[1.] The ultimate period of $\{A^{\lmult t}\lmult x\}$ divides the cyclicity $\gamma$ of the critical graph $\crit(A)$.
\item[2.] The ultimate period of $\{A^{\lmult t}\}$ is equal to $\gamma$.
\end{itemize}
\end{theorem}

\begin{proof}
1.: Follows from Theorem~\ref{CSR-luk}.
2.: Observe that $A\geq A^{[1]}$, and that $(A^{\lmult t})^{[1]}=(A^{[1]})^{\lmult t}$. Also,
the {\L}ukasiewicz product of two $\{0,1\}$ matrices coincides with their product in the
Boolean algebra. This implies that if $\rad(A)=1$ (i.e., the ``graph of ones'' has
nontrivial strongly connected components), then $(A^{\lmult t})^{[1]}$ is nonzero for any $t$,
and so is $A^{\lmult t}$. As the period of $(A^{\lmult t})^{[1]}=(A^{[1]})^{\lmult t}$ is $\gamma$,
the period of $A^{\lmult t}$ cannot be less. Equation~\eqref{CSR-mats}, expressing the ultimate powers
of $A$ in the $CSR$ form, assures that it cannot be more than $\gamma$.
\end{proof}

\begin{proposition}
The exact ultimate period of $\{A^{\lmult t}\lmult x\}$ can be computed in
$O(n^3\log n)$ time.
\end{proposition}
\begin{proof}
First we find $\rad(A)$, in no more than $O(n^3)$ time.
If $\rad(A)<1$ then all orbits convege to $\bzero$. Otherwise,
by Theorem~\ref{CSR-luk}, the subvector of $A^{\otimesluk t}\otimesluk x$
extracted from the first $c$ indices equals $S^tR(x\oplus\bzero)$. More generally,
\begin{equation}
\begin{split}
A^{\otimesluk t}\otimesluk x&=CS^tRx\oplus\bzero=\\
&=C(S^tR(x\oplus\bzero))\oplus\bzero,
\end{split}
\end{equation}
since $CS^tR\bzero\leq\bzero$ (observe that none of the entries of $\amine$ and hence any of its
powers exceed zero).  This shows that the last $(n-c)$ components of $A^{\otimesluk t}\otimesluk x$
are tropical affine combinations of the first (critical) coordinates, and hence we only need to determine
the period of $S^tR(x\oplus\bzero)$, which is the ultimate period of
the first $c$ coordinates of $(\amine)^t(x\oplus\bzero)$. The latter period can be computed
in $O(n^3\log n)$ time by means of an algorithm described in~\cite{Ser-09}, see also
\cite{But:10,HA-99}.
\end{proof}

Butkovi\v{c}~\cite{But:10} and Sergeev~\cite{Ser-11} studied the so-called attraction cones in tropical
algebra, that is, sets of vectors $x$ such that the orbit $A^tx$ hits an eigenvector of $A$ at some $t$.

In the case of max-{\L}ukasiewicz algebra, {\bf attraction sets} can be defined similarly.
Then either $\rho(A)<1$ and then all orbits converge to $0$, or $\rho(A)=1$ and then there may be a non-trivial periodic regime, in which the non-critical components of $A^{\otimesluk t}\otimesluk x$ are tropical affine
combinations of the critical ones. It follows that (like in max algebra) the convergence of
$A^{\otimesluk t}\otimesluk x$ to an eigenvector with eigenvalue $1$ is determined by critical components only.
Since these are given by $S^tR(x\oplus\bzero)$, the attraction sets in max-{\L}ukasiewicz algebra are solution sets to
\begin{equation}
\label{attr-sys}
Rx\oplus\bzero=SRx\oplus\bzero,\ 0\leq x_i\leq 1\forall i.
\end{equation}
This system of equations can be analysed as in~\cite{Ser-11}, a task that we postpone to a future
work.

{\bf Example.} Let us consider the following matrix
\begin{equation}
A=
\begin{pmatrix}
0.62  & 1.00  &  0.57 &  0.14\\
1.00 & 0.18 & 0.17 & 0.18\\
0.38 & 0.59 & 0.65 & 0.43\\
0.10 & 0.18 & 0.25 & 0.33
\end{pmatrix}
\end{equation}

Its max-{\L}ukasiewicz powers proceed as follows:

\begin{equation*}
\begin{split}
A^{\otimes_L 2}&=
\begin{pmatrix}
1 & 0.62 &  0.22 & 0.18\\
0.62 & 1 &  0.57 & 0.14\\
0.59 & 0.38 &  0.30 & 0.08\\
0.18 & 0.10 &  0   & 0
\end{pmatrix},\quad
A^{\otimes_L 3}=
\begin{pmatrix}
0.62 &  1 &  0.57 &  0.14\\
1 &  0.62 &  0.22 &  0.18\\
0.38 &  0.59 &  0.16 & 0\\
0.10 &  0.18 &  0 &  0
\end{pmatrix}.\\
A^{\otimes_L 4}&=
\begin{pmatrix}
1 & 0.62 & 0.22 &  0.18\\
0.62 & 1 & 0.57 &  0.14\\
0.59 & 0.38 &  0 & 0\\
0.18 & 0.10 &  0 & 0
\end{pmatrix},\quad
A^{\otimes_L 5}=
\begin{pmatrix}
0.62  &   1  &  0.57 &  0.14\\
1  &  0.62  &  0.22  & 0.18\\
0.38 & 0.59 &  0.16 &  0\\
0.10 & 0.18 &   0 & 0
\end{pmatrix}
\end{split}
\end{equation*}

We see that the sequence $\{A^{\otimes_L t}\}_{t\geq 1}$ is
ultimately periodic with period $2$, and the periodicity starts at
$t=3$. The term $(\amine)^{t-1}A=CS^{t-1}RA$ of~\eqref{convert-mats}
dominates in all the columns and rows of $A^{\otimes_L t}$ with
critical indices $(1,2)$. In the periodic regime $(t\geq 3)$, the
zero term dominates in almost all entries of the non-critical
$A^{\otimes_L t}_{MM}$ where $M=\{3,4\}$.

\section{Acknowledgement}

We would like to thank Dr. Imran Rashid for useful discussions that
have led to the idea of using the max-plus linear algebra and
tropical convexity in this context. We also thank the anonymous referee for a number of
useful comments.

\if{
\section{Discussion and further research}

We are led to the following questions:

{\bf 1. {\L}ukasiewicz eigenvectors.} Equation~\eqref{mp-spectral}
defines a tropical polyhedron, more precisely, a compact tropically
convex set with a finite number of extremal points.\\
$\bullet$ Design an algorithm for computing the extremal points (or
adapt the known methods of Allamigeon et al.~\cite{All:09,AllGG-10}.

{\bf 2. {\L}ukasiewicz matrix powers and orbits.}  We have shown
that the powers of matrices in {\L}ukasiewicz algebra are strongly
related to tropical algebra and are simpler than their tropical
counterpart. Generically they fall to zero after a finite
number of steps, and unfortunately, so do Sarah's funds in the aggressive network.\\
$\bullet$ If $\rho(A)<1$ then compute or estimate the time when
$A^{\otimes_L t}=\bzero$, and if $\rho(A)=1$ then give the bounds on
periodicity threshold.
}\fi


\end{document}